\documentclass[10pt]{article}

\usepackage{enumitem} 
\usepackage{amsthm,amsfonts}
\usepackage{amsmath,amssymb,wasysym}
\usepackage{graphicx}
\usepackage{wrapfig}
\usepackage{comment}
\usepackage{caption}
\usepackage{subcaption}
\usepackage{tikz}
\usetikzlibrary{arrows.meta}
\usepackage{breqn}
\usepackage{booktabs}
\usepackage{authblk}
\usepackage{layout}
\usepackage[makeroom]{cancel}

\usepackage{pgfplots}
\pgfplotsset{compat=1.18}

\usepackage[%
  colorlinks=true,%
	linkcolor	=blue,%
	citecolor =blue,%
  urlcolor=blue%
]{hyperref}

\newcommand{\ControlC}{\textsf{\textup{\textbf{c}}}}
\newcommand{\ControlS}{\textsf{\textup{\textbf{s}}}}

\usepackage{hypcap}%importante
\usepackage{float}% Antes nofloat importante

\usepackage[left=2.8cm,right=2.8cm, top=3.1cm, bottom=2.8cm]{geometry}
\newtheorem{theo}{Theorem}[section]
\newtheorem{lemm}[theo]{Lemma}

\newtheorem{prop}[theo]{Proposition}

\newtheorem{remark}[theo]{Remark}
\usepackage{stackengine}%para hacer inmersión compacta
\newcommand{\R}{\hbox{\rm I \kern -5pt R}}     % simbolo de Reales
\newcommand{\p} {\hbox{\rm I \kern -5pt P}}

\def\x        {\textit{\textbf{x}}}
\def\n        {\nabla}

\numberwithin{equation}{section}

\definecolor{ao(english)}{rgb}{0.0, 0.5, 0.0}
\definecolor{aqua}{rgb}{0.0, 1.0, 1.0}
\definecolor{bazaar}{rgb}{0.6, 0.47, 0.48}
\definecolor{britishracinggreen}{rgb}{0.0, 0.26, 0.15}

\title{Direct and indirect operator learning strategies for parametric optimal control problems}

\author[a]{Juan J. Forero-Herna\'ndez\thanks{E-mail: \href{mailto: juan2258036@correo.uis.edu.co}{ juan2258036@correo.uis.edu.co}}}
\author[b]{Francisco Guill\'en-Gonz\'alez\thanks{E-mail (Corresponding author): \href{mailto:guillen@us.es}{guillen@us.es}}}
\author[a]{\'Elder J. Villamizar-Roa\thanks{E-mail: \href{mailto:jvillami@uis.edu.co}{jvillami@uis.edu.co}}}

\affil[a]{Universidad Industrial de Santander, Escuela de Matem\'{a}ticas,  Carrera 29 Calle 9, A.A. 678, CP 680002, Bucaramanga, Colombia.}
\affil[b]{Universidad de Sevilla, Dpto. Ecuaciones Diferenciales y Análisis Numérico and IMUS, Facultad de Matemáticas, C. Tarfia, s/n, 41012, Sevilla, España.}

\date{}

\begin{document}
\maketitle

%%%%%%%%%%%%%%%%%%%%%%%%%%%%%%%%%%%%%%%%%%%%%%%%%%%%%%%%%%%%%
%%%%%%%%%%%%%%%%%%%%%%%%%%%%%%%%%%%%%%%%%%%%%%%%%%%%%%%%%%%%%
% RESUMEN
%%%%%%%%%%%%%%%%%%%%%%%%%%%%%%%%%%%%%%%%%%%%%%%%%%%%%%%%%%%%%
%%%%%%%%%%%%%%%%%%%%%%%%%%%%%%%%%%%%%%%%%%%%%%%%%%%%%%%%%%%%%

\begin{abstract}
This paper focuses on the use of learning nonlinear operators via Deep Operator Networks (DeepONets) for approximating the solutions of optimal control problems. We study two operator learning strategies, namely, indirect and direct, both based on DeepONets, for parametric families of optimal control problems where the same problem must be solved many times for different initial data, different parameters, different cost function weights, and different bounds defining the set of the admissible controls. In the indirect strategy two DeepONets approximate the solution operators of the state and adjoint systems, and their predictions are used inside iterative projected gradient algorithms for finding the optimal control. In the direct strategy, a single DeepONet approximates the operator that maps directly to the optimal control generated by a reference solver; consequently, a new control is obtained through a single evaluation of the network. In order to analyze the computational cost, the efficiency, the speed and the accuracy of the learning strategies, we consider an optimal control problem related to an ordinary differential system that describes the temporal evolution of the glioblastoma growth subject to therapeutic interventions. For this specific optimal control problem we propose a discretize-then-optimize strategy, based on an implicit-explicit finite difference method, and then, we develop several numerical experiments in order to compare the two DeepONet learning strategies for this problem. The proposed operator learning DeepONet strategies can be adapted to approximate large classes of optimal control problems related to diverse differential systems. 

\vspace{0.3cm}

\vspace{0.3cm}

%\noindent{\bf AMS subject classifications.} 35K55; 35Q35; 35Q92; 92C17
\end{abstract}

%%%%%%%%%%%%%%%%%%%%%%%%%%%%%%%%%%%%%%%%%%%%%%%%%%%%%%%%%%%%%
%%%%%%%%%%%%%%%%% %%%%%%%%%%%%%%%%%%%%%%%%%%%%%%%%%%%%%%%%%%%%
% CAPÍTULO 1: INTRODUCCIÓN
%%%%%%%%%%%%%%%%%%%%%%%%%%%%%%%%%%%%%%%%%%%%%%%%%%%%%%%%%%%%%
%%%%%%%%%%%%%%%%%%%%%%%%%%%%%%%%%%%%%%%%%%%%%%%%%%%%%%%%%%%%%
\section{Introduction}
In many applications, one must solve not just a single optimal control problem, but a family of problems determined by the variability of the initial conditions, the parameters of the model, the objective functional weights, or control constraints. This situation arises, for example, when evaluating different scenarios, conducting sensitivity analyses, or seeking to tailor an intervention to the characteristics of each individual.

Numerical methods based on gradient descent allow these problems to be addressed systematically: for a given control, the state system is first solved forward in time, followed by the adjoint system backward in time, and a descent direction is calculated from both solutions. This procedure is repeated until a stopping criterion is met \cite{Hinze,Troltzsch}. Although the use of the adjoint avoids the need to compute separately the sensitivities with respect to each control component, the repeated solution of the state and adjoint systems can be computationally expensive when the same problem must be evaluated for a large number of parametric configurations.\\

In recent years, neural networks have been incorporated into the study of problems governed by differential equations. Among the most widely used approaches are physics-informed neural networks (PINNs), which incorporate model equation residuals and initial or boundary conditions into the loss function \cite{Raissi}. This methodology has been extended to optimal control problems using networks designed to approximate state variables, control variables, and, in some formulations, the adjoint variables \cite{Mowlavi2023,BarryStraume2026}. Subsequently, methods capable of handling parametric families of these problems have been developed. For instance, Demo {\it et al.}\cite{Demo2023} considered the parameters of the model  as additional network inputs and investigated the application of this formulation to optimal control problems related to PDEs. In addition, Yin {\it et al.} \cite{Yin2024} proposed an adjoint-driven method in which distinct networks simultaneously approximate the state, adjoint, and control variables for a family of parameter values. These approaches enable the derivation of solutions for new parameter configurations without solving each problem using a conventional numerical method. However, their training explicitly utilizes the system equations and optimality conditions, thereby tying the networks to the specific control problem for which they were constructed.\\

Operator learning offers another possibility. Instead of approximating a single function, its goal is to approximate mappings between function spaces. The universal approximation results for nonlinear operators by Chen and Chen \cite{ChenChen} provide a theoretical structure of this approach. Then, Lu {\it et al.} \cite{LuDeepONet} proposed Deep Operator Networks (DeepONets), whose architecture combines a \textit{branch} network, responsible for representing the input function via its values at a set of sensors, and a \textit{trunk} network, which takes the point where the output function is to be evaluated. After a training phase, a DeepONet can be evaluated for new input functions without solving the corresponding differential problem. 
Recently, the application of these ideas to solve optimal control problems has taken various directions. Hwang {\it et al.} \cite{Hwang2022} approximated the solution operator of the state equation and used the learned surrogate in a subsequent optimization stage with respect to the control. Yong {\it et al.} \cite{Yong2024} proposed multi-input, multi-output operator networks to approximate the variables associated with the optimality system.\\

In this work, we consider a parametric family of optimal control problems determined by the data appearing in the state system, the weights of the objective functional, and the constraints defining the set of admissible controls. Based on this formulation, we identify three operators: the state operator, the adjoint operator, and the operator that maps each data set to the control produced by a reference algorithm. Note that the last operator is well-defined even when the problem possesses multiple minimizers, as it represents not the entire set of solutions, but rather the specific element selected by a pre-established numerical procedure. To this end, we propose two supervised strategies based on DeepONets, namely, \textbf{\textit{Indirect strategy}} and \textbf{\textit{Direct strategy}}. In the indirect strategy two independent networks are trained to approximate the state and adjoint operators. Their predictions replace the numerical solutions of the state and adjoint systems within the gradient descent method. On the other hand, in the  direct strategy a single DeepONet is trained to approximate the optimal control operator. In this case, a new control is obtained through a single evaluation of the network, but its admissibility depends on the accuracy with which the network has learned the constraints present in the data. 
This operator is only defined on the data for which the optimal control algorithm converges.\\

These two learning strategies are analyzed in the context of an optimal control problem associated with a controlled ODE model describing glioma cell growth in the presence of oxygen. In this model, the controls represent cytotoxic and antiangiogenic therapies, and the objective is to reduce the tumor density and regulate the oxygen concentration while limiting the side effects caused by these therapeutic interventions. Additionally, a set of admissible controls subject to pointwise and integral constraints is considered (see the mathematical formulation in Section \ref{sec:ejemplo}). The aim is to use this control problem to validate the Indirect and Direct learning strategies. In order to establish a common frame of reference for this specific optimal control problem, we develop a {\it discretize-then-optimize} scheme. First, the state system and the objective functional are discretized via an implicit-explicit finite difference method; subsequently, the sensitivity system, the adjoint system, the reduced functional, and its gradient are derived at the discrete level. The same method is used to generate the supervised data, construct the labels for the direct strategy, and serve as the reference against which the two learning strategies are evaluated.\\

The study of an optimal control problem associated with an ODE model provides a particularly demanding scenario for evaluating operator learning. Since the system is low-dimensional, the reference {\it discretize-then-optimize} scheme is explicit and is highly efficient; indeed, each time step of the state and adjoint schemes reduces to just a few arithmetic operations, meaning the potential advantage of a neural network-based approach is minimal.\\

In summary, the contributions of this paper are the following: First, we formulate two operator learning strategies based on DeepONets, for the analysis of parametric families of optimal control problems. These strategies are supported by the universal approximation theorem for operator (Theorem~\ref{teo:aproximacion_operadores}), and have the advantage of being very efficient during the online stage, since the online evaluation only requires forward passes of standard feed-forward networks and no specialized numerical library. In addition, the main general advantage of the proposed indirect strategy is that it preserves the structure of the control algorithm, for instance the projection onto the admissible set, so that the computed controls are always admissible. On the other hand, the main advantage of the direct strategy is that it avoids the use of the control algorithm, with a very low computational cost. Second, we analyze an optimal control problem related to an ODE model of glioblastoma growth influenced by the oxygen. The  model considers a pair of controls describing the chemotherapy and antiangiogenic therapy. The cost functional aims to reduce the tumor density, bring the oxygen concentration close to a desired value, and penalize the use of therapies. We analyze the optimal control problem theoretically (well-posedness of the state system, existence of optimal controls and discrete firs-order optimality conditions) and solve it numerically, and then, we compare that numerical approximation with the counterparts obtained with the two learning strategies.\\

The outline of this paper is as follows. In Section 2, we establish the notation and necessary background on neural networks and DeepONets. In Section 3, we formulate the parametric family of optimal control problems, we describe the classical reference method, and  finally, we introduce the indirect and direct learning strategies based on DeepONet. In Section 4, we analyze the ODE-based optimal control problem for cytotoxic and antiangiogenic therapies, and develop the discretize-then-optimize strategy; we also describe the construction of datasets and the network training. Finally, in the same section, we compare numerically the proposed methodologies.

\section{Preliminaries of DeepONets}

In this section we collect the notation and some preliminaries that will be used throughout the paper.

\subsection{Notation}\label{sec:notacion}

Throughout this paper, $T>0$ denotes a fixed final time. For $1\leq p\leq\infty$, we denote by $L^p(0,T)$ the usual Lebesgue space of real-valued functions on $(0,T)$, endowed with the norm $\|\cdot\|_{L^p}$. The corresponding spaces of vector-valued functions are denoted by $L^p(0,T;\mathbb{R}^d)$. In particular, the inner product and norm in $L^2(0,T)$ are denoted by $(\cdot,\cdot)_{L^2}$ and $\|\cdot\|_{L^2}$, respectively. If $K$ is a compact subset of a normed space, $C(K)$ denotes the Banach space of real-valued continuous functions on $K$, endowed with the supremum norm $\|f\|_{C(K)}:=\max_{x\in K}|f(x)|$. The Euclidean norm in $\mathbb{R}^d$ is denoted by $\|\cdot\|_2$.\\

Let $H$ be a real Hilbert space, with inner product $(\cdot,\cdot)_H$ and norm $\|\cdot\|_H$, and let $A\subset H$ be a nonempty, closed and convex set. We denote by $\operatorname{Proj}_A:H\to A$ the orthogonal projection onto $A$, defined by
\[
\operatorname{Proj}_A(v) := \underset{w\in A}{\operatorname{arg\,min}}\|v-w\|_H, \qquad v\in H.
\]
In the particular case $H=\mathbb{R}$ and $A=[a,b]$, with $a\leq b$, this projection is given by
\[
\operatorname{Proj}_{[a,b]}(r) = \min\{\max\{r,a\},b\}, \qquad r\in\mathbb{R}.
\]

\subsection{Artificial neural networks (ANNs)}

The Artificial neural networks (ANNs) are function approximators constructed by composing affine mappings with nonlinear activation functions. The basic unit is the \textit{simple perceptron}, in which, given an input vector $x=[x_1,\dots,x_d]^{T}\in\mathbb{R}^d$, a weight vector $\mathbf{w}=[w_1,\dots,w_d]^{T}\in\mathbb{R}^d$, and a bias $b\in\mathbb{R}$, the unit computes the affine combination $z=\mathbf{w}^{T}x+b$ and returns $y=\varphi(z)$, where $\varphi$ is a nonlinear function called the \textit{activation function}.\\

A \textit{multilayer perceptron} (MLP) is obtained by composing several layers of such units. Let $D\geq2$ be the depth of the network,  $n_0$ the dimension of the input, $n_D=n_{out}$ the number of outputs and $n_1,\dots,n_{D-1}\in\mathbb{N}$ denote the widths of the hidden layers. The set of trainable parameters is denoted by $\theta:=((W^{(l)},b^{(l)}))_{l=1}^{D}$, where
\[
W^{(l)}\in\mathbb{R}^{n_l\times n_{l-1}},\qquad b^{(l)}\in\mathbb{R}^{n_l},\qquad l=1,\dots,D.
\]
Then, given $x\in\mathbb{R}^{n_0}$, the network is defined recursively by $a^{(0)}(x):=x$ and
\[
z^{(l)}(x):=W^{(l)}a^{(l-1)}(x)+b^{(l)},\qquad a^{(l)}(x):=\varphi_l\bigl(z^{(l)}(x)\bigr),\qquad l=1,\dots,D-1,
\]
where $\varphi_l$ denotes the activation function of the $l$-th hidden layer, acting component-wise, while the output layer is given by
\[
z^{(D)}(x):=W^{(D)}a^{(D-1)}(x)+b^{(D)},\qquad \mathcal{N}_\theta(x):=z^{(D)}(x)\in\mathbb{R}^{n_{out}}.
\]
The parameters $\theta$ are determined by minimizing a loss function, which is usually used by means of a gradient descent method; in this paper we employ the Adam optimizer \cite{Adam}, and the gradients are computed by backpropagation \cite{Goodfellow} with respect to $\theta$.

\subsection{Deep operator networks}

In many problems of scientific computing, one is not only interested in approximating a particular solution of a differential equation, but also in understanding how the solution changes when the data of the problem are modified, for instance, an initial condition, a parameter, a function or a parametric control. Achieving this using classical methods would require repeating the entire numerical calculation, which is often costly in terms of time and processing resources.\\

Deep operator networks (DeepONets) \cite{LuDeepONet} are emerging as a machine learning-based alternative for approximating operators between functional spaces. Unlike a standard neural network, which approximates a mapping between finite-dimensional spaces, a DeepONet is designed to learn mappings of the form
\[
\mathcal{G}:\mathcal{X}\longrightarrow\mathcal{Y},
\]
where $\mathcal{X}$ and $\mathcal{Y}$ denote the spaces of input and output functions, respectively. For example, these neural operators can approximate the solution operators associated with differential equations; indeed, under appropriate assumptions, the data of an evolution problem uniquely determine its solution, and the mapping that sends the data to the solution defines an operator between function spaces.\\

The central idea of a DeepONet is to approximate the (possibly nonlinear) operator $\mathcal{G}$ by the cooperation of two neural networks: a \textit{branch net}, which processes the input function $f$, and a \textit{trunk net}, which processes the evaluation point $x$ \cite{LuDeepONet}. Since the domain of $f$ is an infinite set, the function $f$ must be represented in a finite way. To this end, a set of $m$ points $\{\xi_1,\dots,\xi_m\}\subset\operatorname{Dom}(f)$, called \textit{sensors}, is fixed and $f$ is encoded by its vector of sampled values
\begin{equation}\label{eq:vector_entrada_deeponet}
\mathbf{f}_\xi=\bigl(f(\xi_1),f(\xi_2),\dots,f(\xi_m)\bigr)\in\mathbb{R}^m.
\end{equation}
It is important to remark that the sensors are common to all input functions and remain fixed during the training; only the values $f(\xi_j)$ change. Thus, the branch net receives $\mathbf{f}_\xi$ and produces a vector of $\ell$ components ($\ell$ being called the \textit{latent dimension})
\[
\bigl(b_1(\mathbf{f}_\xi),b_2(\mathbf{f}_\xi),\dots,b_\ell(\mathbf{f}_\xi)\bigr)\in\mathbb{R}^\ell,
\]
whereas the trunk net receives the point $x$ at which the operator is to be evaluated and produces a vector of $\ell$ components
\[
\bigl(t_1(x),t_2(x),\dots,t_\ell(x)\bigr)\in\mathbb{R}^\ell.
\]
The approximation of the operator is obtained by combining both outputs by means of an inner product, plus a scalar bias $b_0\in\mathbb{R}$, in the following way:
\begin{equation}\label{eq:prediccion_deeponet}
\mathcal{G}_\vartheta[\mathbf f_\xi](x)=\sum_{i=1}^{\ell}b_i\bigl(\mathbf{f}_\xi\bigr)\,t_i(x)+b_0.
\end{equation}
The set of trainable parameters of the DeepONet, denoted by $\vartheta$, is composed of the union of the weights and biases of the branch net ($\theta_b$) and of the trunk net ($\theta_t$), together with the bias $b_0$, so that $\vartheta=(\theta_b,\theta_t,b_0)$. Figure \ref{fig:deeponet} summarizes this construction.

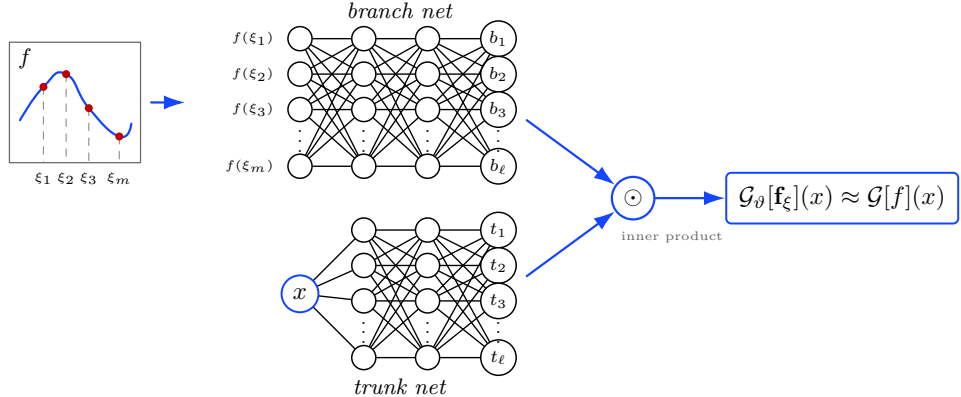
\begin{figure}[H]
\centering
\resizebox{0.8\linewidth}{!}{
\begin{tikzpicture}[x=1cm,y=1cm,line cap=round,line join=round]
  \definecolor{blueacc}{RGB}{20,70,255}
  \tikzset{
    flow/.style={blueacc,-{Latex[length=3mm,width=2mm]}, line width=0.9pt},
    netnode/.style={circle, draw=black, line width=0.6pt, fill=white, minimum size=0.34cm, inner sep=0pt},
    outnode/.style={circle, draw=black, line width=0.6pt, fill=white, minimum size=0.5cm, inner sep=0pt, font=\scriptsize},
    innode/.style={circle, draw=blueacc, line width=0.8pt, fill=white, minimum size=0.52cm, inner sep=0pt},
    netline/.style={draw=black, line width=0.55pt},
    tag/.style={font=\small\itshape},
  }

  %--------------------------------------------------------
  % Input function f sampled at the m sensors
  %--------------------------------------------------------
  \draw[black!60, line width=0.5pt] (-5.1,0.50) rectangle (-3.25,2.20);
  \node[font=\small] at (-4.88,1.97) {$f$};
  \draw[blueacc, line width=0.9pt] plot [smooth, tension=0.9]
    coordinates {(-4.95,1.10) (-4.62,1.57) (-4.3,1.75) (-3.98,1.27) (-3.55,0.87) (-3.38,1.05)};
  \foreach \sx/\sy/\lab in {-4.62/1.57/{\xi_1}, -4.3/1.75/{\xi_2}, -3.98/1.27/{\xi_3}, -3.55/0.87/{\xi_m}}{
    \draw[black!45, dashed, line width=0.45pt] (\sx,\sy) -- (\sx,0.50);
    \fill[red!75!black] (\sx,\sy) circle (1.6pt);
    \node[font=\tiny] at (\sx,0.28) {$\lab$};
  }

  % sampling -> branch
  \draw[flow] (-3.1,1.35) -- (-2.62,1.35);

  %--------------------------------------------------------
  % Branch net (encodes f)
  %--------------------------------------------------------
  \begin{scope}[shift={(0,1.35)}]
    \foreach \y/\n in {0.9/1,0.4/2,-0.1/3,-0.9/4}{
      \node[netnode] (bi\n) at (-1.0,\y) {};
      \node[netnode] (bh\n) at (-0.1,\y) {};
      \node[netnode] (bg\n) at (0.8,\y) {};
    }
    \node[outnode] (bo1) at (1.8,0.9) {$b_1$};
    \node[outnode] (bo2) at (1.8,0.4) {$b_2$};
    \node[outnode] (bo3) at (1.8,-0.1) {$b_3$};
    \node[outnode] (bo4) at (1.8,-0.9) {$b_\ell$};
    \foreach \x in {-1.0,-0.1,0.8,1.8}{\node[font=\scriptsize] at (\x,-0.44) {$\vdots$};}
    \foreach \i in {1,...,4}{\foreach \j in {1,...,4}{
      \draw[netline] (bi\i) -- (bh\j);
      \draw[netline] (bh\i) -- (bg\j);
      \draw[netline] (bg\i) -- (bo\j);}}
    \foreach \y/\lab in {0.9/{f(\xi_1)},0.4/{f(\xi_2)},-0.1/{f(\xi_3)},-0.9/{f(\xi_m)}}{
      \node[font=\tiny, anchor=east] at (-1.25,\y) {$\lab$};
    }
  \end{scope}
  \node[tag] at (0.4,2.65) {branch net};

  %--------------------------------------------------------
  % Trunk net (encodes the evaluation point x)
  %--------------------------------------------------------
  \begin{scope}[shift={(0,-1.35)}]
    \node[innode] (tx) at (-1.0,0) {$x$};
    \foreach \y/\n in {0.9/1,0.4/2,-0.1/3,-0.9/4}{
      \node[netnode] (th\n) at (-0.1,\y) {};
      \node[netnode] (tg\n) at (0.8,\y) {};
    }
    \node[outnode] (to1) at (1.8,0.9) {$t_1$};
    \node[outnode] (to2) at (1.8,0.4) {$t_2$};
    \node[outnode] (to3) at (1.8,-0.1) {$t_3$};
    \node[outnode] (to4) at (1.8,-0.9) {$t_\ell$};
    \foreach \x in {-0.1,0.8,1.8}{\node[font=\scriptsize] at (\x,-0.44) {$\vdots$};}
    \foreach \n in {1,...,4}{\draw[netline] (tx) -- (th\n);}
    \foreach \i in {1,...,4}{\foreach \j in {1,...,4}{
      \draw[netline] (th\i) -- (tg\j);
      \draw[netline] (tg\i) -- (to\j);}}
  \end{scope}
  \node[tag] at (0.4,-2.65) {trunk net};

  %--------------------------------------------------------
  % Inner product and output
  %--------------------------------------------------------
  \node[circle, draw=blueacc, line width=0.9pt, minimum size=0.6cm, inner sep=0pt]
    (dot) at (3.7,0) {$\odot$};
  \node[font=\tiny, black!60] at (4.25,-0.55) {inner product};

  \draw[flow] (2.2,1.10) -- (dot);
  \draw[flow] (2.2,-1.10) -- (dot);

  \node[rounded corners=2pt, draw=blueacc, line width=0.8pt, inner sep=5pt, anchor=west]
    (out) at (5.0,0) {$\mathcal{G}_\vartheta[\mathbf{f}_\xi](x)\approx\mathcal{G}[f](x)$};
  \draw[flow] (dot) -- (out);

\end{tikzpicture}
}
\caption{Structure of a DeepONet: the \emph{branch} net encodes the input function $f$ sampled at the $m$ sensors $\{\xi_1,\dots,\xi_m\}$, and the \emph{trunk} net encodes the evaluation point $x$; the inner product of their outputs, of latent dimension $\ell$, produces the approximation $\mathcal{G}_\vartheta [\mathbf{f}_\xi](x)$ of $\mathcal{G}[f](x)$.}
\label{fig:deeponet}
\end{figure}

\begin{remark}[Vector-valued outputs]\label{rem:multisalida}
The equation \eqref{eq:prediccion_deeponet} produces a single real value. When the operator to be approximated has $r$ components, that is, when
\[
\mathcal{G}[f](x)=\bigl(\mathcal{G}^{1}[f](x),\dots,\mathcal{G}^{r}[f](x)\bigr),
\]
a possible approach is to let the branch net produce $r\ell$ coefficients, divided into $r$ latent vectors, while the trunk net produces a shared vector of dimension $\ell$. Each output is obtained from the corresponding inner product, plus a trainable bias. Several input functions and scalar data can also be included by concatenating them into the branch input.
\end{remark}

\subsection{Supervised training of a DeepONet}

In the supervised training of a DeepONet one uses a set of input--output pairs. To this end, it is important to distinguish four elements \cite{Boulle, LuDeepONet}:
\begin{itemize}
    \item[\textbf{i)}] The \textit{sensors} $\{\xi_1,\dots,\xi_m\}$, at which each input function is sampled.
    \item[\textbf{ii)}] The \textit{input functions} $\{f^{(1)},\dots,f^{(N_{\mathrm{tr}})}\}$, which enter the network evaluated at the sensors; for each one the vector $\mathbf{f}_{\xi}^{(i)}=(f^{(i)}(\xi_1),\dots,f^{(i)}(\xi_m))$ is formed.
    \item[\textbf{iii)}] The \textit{query points} or \textit{observation points} $\{x_1,\dots,x_{N_{\mathrm{ob}}}\}$, at which the output is evaluated. These points belong to the domain of the functions of the output space.
    \item[\textbf{iv)}] The output values $\{\mathcal{G}[f^{(i)}](x_j)=u^{(i)}(x_j)\}$, corresponding to some input function $f^{(i)}$ and evaluated at some observation point $x_j$.
\end{itemize}
Therefore, the training dataset is
\[
\mathcal{D}_{\mathrm{tr}}=\Big\{(\mathbf{f}_{\xi}^{(i)},x_j,u^{(i)}(x_j)):\ i=1,\dots,N_{\mathrm{tr}},\ j=1,\dots,N_{\mathrm{ob}}\Big\},
\]
where $N_{\mathrm{tr}}$ is the number of input functions and $N_{\mathrm{ob}}$ is the number of query points per function.\\

The generation of these data consists of two phases. In the first one, the input functions $f^{(i)}$, $i=1,\dots,N_{\mathrm{tr}}$, are sampled from a distribution or from a previously defined family of functions. This choice must represent adequately the class of inputs on which the operator will be approximated. Once each function $f^{(i)}$ has been generated, it is evaluated at the set of sensors $\{\xi_1,\dots,\xi_m\}$, which produces the finite representation \eqref{eq:vector_entrada_deeponet} received by the branch net. Consequently, the number and the distribution of these sensors must be chosen in such a way that they allow one to represent adequately the functions belonging to the considered input family. In the second phase, for each function $f^{(i)},$ the corresponding output evaluated at the observation points $x_j$,
\[
u^{(i)}(x_j):=\mathcal{G}[f^{(i)}](x_j),
\]
is computed. When $\mathcal{G}$ is the solution operator associated with a differential equation, this stage requires to solve the problem for each one of the generated input functions by means of a numerical method.\\

A usual choice in order to quantify the difference between the predictions of the network and the reference solutions is the mean squared error. This (supervised) loss function is then defined as
\begin{equation}\label{eq:perdida_supervisada_deeponet}
    \mathcal{L}(\vartheta):=\frac{1}{N_{\mathrm{tr}}N_{\mathrm{ob}}}\sum_{i=1}^{N_{\mathrm{tr}}}\sum_{j=1}^{N_{\mathrm{ob}}}\bigl\Vert\mathcal{G}_{\vartheta}[\mathbf{f}_{\xi}^{(i)}](x_j)-u^{(i)}(x_j)\bigr\Vert_2^{2},
\end{equation}
and the training of the DeepONet is formulated by means of the optimization problem
\begin{equation}\label{eq:problema_optimizacion_deeponet}
    \vartheta^{\ast}\in\operatorname{arg}\underset{\vartheta}{\min}\ \mathcal{L}(\vartheta).
\end{equation}
In practice, \eqref{eq:problema_optimizacion_deeponet} is solved by a gradient descent method in which, at each optimization step, the sum in \eqref{eq:perdida_supervisada_deeponet} is restricted to a mini-batch of sample--point pairs. The theoretical justification of the architecture \eqref{eq:prediccion_deeponet} is given by the following universal approximation theorem for operators, proved in \cite{ChenChen}.

\begin{theo}[Universal approximation theorem for operators \cite{ChenChen}]\label{teo:aproximacion_operadores}
Assume that $\varphi$ is a continuous non-polynomial function, $X$ is a Banach space, $K_1\subset X$ and $K_2\subset\mathbb{R}^d$ are two compact sets in $X$ and $\mathbb{R}^d$, respectively, $V$ is a compact set in $C(K_1)$ and $\mathcal{G}:V\to C(K_2)$ is a nonlinear continuous operator. Then, for every $\varepsilon>0$ there exist positive integers $n$, $\ell$, $m$, constants $c_i^k,\ \omega_{ij}^k,\ \theta_i^k,\ \zeta_k\in\mathbb{R}$, $w_k\in\mathbb{R}^d$, $\xi_j\in K_1$, $i=1,\dots,n$, $k=1,\dots,\ell$, $j=1,\dots,m$, such that
\[
\left|\mathcal{G}[f](x)-\sum_{k=1}^{\ell}\underbrace{\sum_{i=1}^{n}c_i^k\varphi\Bigl(\sum_{j=1}^{m}\omega_{ij}^k f(\xi_j)+\theta_i^k\Bigr)}_{\text{branch}}\underbrace{\varphi(w_k\cdot x+\zeta_k)}_{\text{trunk}}\right|<\varepsilon
\]
holds for all $f\in V$ and $x\in K_2$.
\end{theo}

%%%%%%%%%%%%%%%%%%%%%%%%%%%%%%%%%%%%%%%%%%%%%%%%%%%%%%%%%%%%%
% 3. OPERATOR LEARNING FOR PARAMETRIC OPTIMAL CONTROL
%%%%%%%%%%%%%%%%%%%%%%%%%%%%%%%%%%%%%%%%%%%%%%%%%%%%%%%%%%%%%
\section{Operator learning strategies for parametric optimal control problems}\label{sec:estrategias}

In this section we formulate, for a general parametric family of optimal control problems, the classical solver that will be used as a reference and the two strategies based on supervised DeepONets that are the object of this paper. The example of Section \ref{sec:ejemplo} is a particular instance of this framework.

\subsection{A parametric family of optimal control problems}\label{sec:problema_general}

Let $T>0$ be a fixed final time and let $\mathcal{Q}$ be a real Hilbert space and $\mathcal{Y}$ and $\mathcal{Z}$ be real Banach spaces, of functions defined on $(0,T)$, which we call the space of controls, the space of states and the space of adjoint states, respectively. Each individual optimal control problem of the family is determined by a finite dimensional datum
\begin{equation}\label{eq:dato_general}
\lambda=[\lambda_{e},\lambda_{J},\lambda_{U}]\in\Lambda\subset\mathbb{R}^{d_e}\times\mathbb{R}^{d_J}\times\mathbb{R}^{d_U},
\end{equation}
where $\Lambda$ is a compact set, $\lambda_e$ contains the data involved in the state equation (initial conditions and model parameters), $\lambda_J$ contains the data involved in the cost functional (weighting coefficients and desired states), and $\lambda_U$ contains the data defining the set of admissible controls. We assume the following hypotheses.

\begin{itemize}
\item[\textbf{(H1)}] \textit{(State equation)} For every $\lambda_U$, the set $\mathcal{U}_{ad}(\lambda_U)\subset\mathcal{Q}$ of admissible controls is nonempty, closed and convex and, for every $q\in\mathcal{U}_{ad}(\lambda_U)$ and every $\lambda_e$, the state equation
\[
E(y,q,\lambda_e)=0
\]
admits a unique solution $y\in\mathcal{Y}$. We denote by
\[
\mathcal{G}_{e}:\ [q,\lambda_e]\longmapsto y=y(q,\lambda_e)
\]
the corresponding \textit{state operator}, which is assumed to be continuous.

\item[\textbf{(H2)}] \textit{(Cost functional)} For a given  functional $J=J(y,q,\lambda_J),$ we consider the associated \textit{reduced functional}
\[
\widehat{J}_{\lambda}(q):=J\bigl(\mathcal{G}_{e}[q,\lambda_e],q,\lambda_J\bigr),
\]
which is assumed to admit directional derivatives at every admissible control along all feasible directions.

\item[\textbf{(H3)}] \textit{(Existence of optimal controls)} For every $\lambda_U$, it is assumed that the problem
\begin{equation}\label{eq:problema_general}
(P_\lambda)\qquad \min_{q\in\,\mathcal{U}_{ad}(\lambda_U)}\ \widehat{J}_{\lambda}(q)
\end{equation}
admits at least one solution.

\item[\textbf{(H4)}] \textit{(Adjoint representation of the gradient)} For every $q\in\mathcal{U}_{ad}$ and every $\lambda$, the adjoint equation
\[
A\bigl(p,\,\mathcal{G}_e[q,\lambda_e],q,\lambda_e,\lambda_J\bigr)=0
\]
admits a unique solution $p\in\mathcal{Z}$; and we denote the corresponding \textit{adjoint operator} by
\[
\mathcal{G}_{a}:\ [q,\lambda_e,\lambda_J]\longmapsto p=p(q,\lambda_e,\lambda_J).
\] Moreover, the gradient of the reduced functional admits the explicit representation
\begin{equation}\label{eq:gradiente_general}
\nabla\widehat{J}_{\lambda}(q)=\Gamma\bigl(\mathcal{G}_e[q,\lambda_e],\,\mathcal{G}_a[q,\lambda_e,\lambda_J],\,q,\,\lambda_J\bigr),
\end{equation}
where $\Gamma$ is an algebraic expression whose evaluation does not require to solve any additional differential equation.
\end{itemize}

See \cite{Hinze, Troltzsch} for the general theory of optimal control problems governed by differential equations, particularly for the sufficient conditions under which \textbf{(H1)}--\textbf{(H4)} hold.\\

Under \textbf{(H1)}--\textbf{(H4)}, if $\overline{q}\in\mathcal{U}_{ad}(\lambda_U)$ is a local optimal control for $(P_\lambda)$, the first-order necessary optimality condition reads
\begin{equation}\label{eq:desigualdad_variacional_general}
\bigl(\nabla\widehat{J}_{\lambda}(\overline{q}),\,q-\overline{q}\bigr)_{\mathcal{Q}}\ \geq\ 0,\qquad\forall\,q\in\mathcal{U}_{ad}(\lambda_U).
\end{equation}

The key point here is that, at each iteration, a descent method requires two functions: the state function and the adjoint function. These functions are the outputs of two operators, $\mathcal{G}_e$ and $\mathcal{G}_a$, whose inputs are the current control and a finite number of scalar values. This is the type of mapping that a DeepONet is designed to approximate.

\subsection{The classical reference solver}\label{sec:algoritmo_clasico}

As a reference method, we consider a descent algorithm which computes the reduced gradient by means of the adjoint representation \eqref{eq:gradiente_general}, and which we will refer to as the \textit{classical method}. Given the data $\lambda$ and an initial control $q^{0}\in\mathcal{U}_{ad}(\lambda_U)$, a sequence $\{q^{k}\}_{k\geq0}\subset\mathcal{U}_{ad}(\lambda_U)$ is constructed by repeating the following steps.\\

\noindent
\textbf{Step 1. (State and adjoint state)} Given the current control $q^k\in\mathcal{U}_{ad}(\lambda_U)$, compute the state $y^k=\mathcal{G}_e[q^k,\lambda_e]$ by solving the state equation forward in time, and compute the adjoint $p^k=\mathcal{G}_a[q^k,\lambda_e,\lambda_J]$ by solving the adjoint equation backward in time.\\

\noindent
\textbf{Step 2. (Cost functional and reduced gradient)} Compute the value of the cost functional $J_k:=\widehat{J}_{\lambda}(q^k)=J(y^k,q^k,\lambda_J)$ and the reduced gradient
\[
g^{k}:=\nabla\widehat{J}_{\lambda}(q^{k})=\Gamma(y^{k},p^{k},q^{k},\lambda_J)
\]
by means of \eqref{eq:gradiente_general}; observe that this second computation only requires the algebraic evaluation of $\Gamma$ at the two functions obtained in Step 1.\\

\noindent
\textbf{Step 3. (Projected descent)}
Choose a step size $\omega_k>0$, possibly by means of a line search,
and define
\begin{equation}\label{eq:actualizacion_general}
q^{k+1}
=
\operatorname{Proj}_{\mathcal U_{\mathrm{ad}}(\lambda_U)}
\bigl(q^k-\omega_k g^k\bigr).
\end{equation}
The projection guarantees that $q^{k+1}$ remains admissible. When a line search is employed, each trial step requires an additional evaluation of $\widehat J_\lambda$ and, consequently, an additional evaluation of the state operator.\\

\noindent
\textbf{Step 4. (Stopping criterion)} Stop whenever a prescribed test detects that the sequence has stabilized, for instance when the value of the cost functional has remained stable during a fixed number of consecutive iterations, or when the norm of the projected gradient has become smaller than a given tolerance. Stop also when a maximum number \texttt{maxIter} of iterations has been reached. Otherwise, set $k\leftarrow k+1$ and return to Step 1.\\

The result of this procedure is an approximate optimal control $q^{k}\in\mathcal{U}_{ad}(\lambda_U)$, along with the corresponding value $\widehat{J}_\lambda(q^k)$. It is important to note that each iterate generated by the equation \eqref{eq:actualizacion_general} is admissible by construction, since the projection onto $\mathcal{U}_{ad}(\lambda_U)$ is applied at every update, and also at each trial of the step size when a line search is employed. We say that the algorithm \textit{converges} for the pair $[q^{0},\lambda]$ when it stops because the stabilization test of Step 4 is fulfilled, that is, before exhausting the \texttt{maxIter} iterations, and we denote by $q^{*}$ the control returned in that case.\\

The stabilization test evaluates only the behavior of the sequence generated by the algorithm. Therefore, the fact that this sequence stabilizes does not guarantee that the control \(q^{*}\) obtained is a local optimum of \((P_\lambda)\), since the algorithm could terminate prematurely. To reduce this risk, we verify the computed control using the following validation test: Let $\delta>0$, let $\texttt{tol}_{\mathrm{op}}>0$, and let $\{v_1,\dots,v_M\}\subset\mathcal{Q}$ be a finite family of test directions, including the descent direction $-g^{*} $, where $g^{*}:=\nabla\widehat{J}_{\lambda}(q^{*})$. We accept the control $q^{*}$ if:
\begin{equation}\label{eq:filtro_general}
\widehat{J}_{\lambda}\Bigl(\operatorname{Proj}_{\mathcal{U}_{ad}(\lambda_U)}\bigl(q^{*}+\delta\,v_i\bigr)\Bigr)\ \geq\ \bigl(1-\texttt{tol}_{\mathrm{op}}\bigr)\,\widehat{J}_{\lambda}(q^{*}),\qquad i=1,\dots,M,
\end{equation}
that is, if none of the $M$ admissible perturbations of magnitude $\delta$ of $q^{*}$ reduces the value of the cost functional. The condition \eqref{eq:filtro_general} would serve as an additional numerical strategy to validate that $q^{*}$ is a local optimal control.\\

Once the numerical methods used in Step 1 have been specified, the criterion for computing \(g^{k}\) and \(\omega_k\), the stopping criterion parameters, the data \(\delta\), \(\texttt{tol}_{\mathrm{op}}\), and \(\{v_1,\ldots,v_M\}\) used in the validation test, the algorithm is completely defined on the set
\[
\Lambda^{*}:=\bigl\{[q^{0},\lambda]:\ q^{0}\in\mathcal{U}_{ad}(\lambda_U),\ \lambda\in\Lambda,\ \text{the algorithm converges for }[q^{0},\lambda]\ \text{ and }q^{*}\text{ satisfies }\eqref{eq:filtro_general}\bigr\},
\]
and we can define the \textit{optimal control operator}
\begin{equation}\label{eq:general_control_operator}
\mathcal{G}_{c}:\ [q^{0},\lambda]\in\Lambda^{*}\longmapsto q^{*}=q^{*}(q^0,\lambda),
\end{equation}
which assigns, to each initial control and to each set of parameters, the control returned by the procedure. In other words, $\mathcal{G}_{c}$ reproduces the response of the optimization algorithm on the set of data for which the algorithm converges and the returned control passes the validation test \eqref{eq:filtro_general}.

\subsection{The indirect strategy}\label{sec:estrategia_indirecta}

The \textit{indirect strategy} is based on two independent DeepONets that approximate the two solution operators appearing in Step 1 of the classical method, that is,
\begin{equation}\label{eq:operadores_indirecta}
\mathcal{G}_{e}:[q,\lambda_e]\longmapsto y,\qquad \mathcal{G}_{a}:[q,\lambda_e,\lambda_J]\longmapsto p.
\end{equation}
In both cases the control $q$ is represented by its values at $m$ sensors of $[0,T]$, and the scalar data $\lambda_e$ and $\lambda_J$ are concatenated to that vector of sampled values, as described in Remark \ref{rem:multisalida}; the trunk net receives the queried time $t\in[0,T]$. Denoting by $\mathcal{G}_{e}^{\vartheta_e}$ and $\mathcal{G}_{a}^{\vartheta_a}$ the corresponding trained networks, the \textit{indirect strategy} consists in running the classical method of Section \ref{sec:algoritmo_clasico} \textit{with Step 1 replaced by the two evaluations}
\[
y^{k}\approx\mathcal{G}_{e}^{\vartheta_e}[q^{k},\lambda_e],\qquad p^{k}\approx\mathcal{G}_{a}^{\vartheta_a}[q^{k},\lambda_e,\lambda_J].
\]
All the remaining steps in the algorithm are unchanged.\\

This strategy has two important consequences. First, the iterations remain exactly admissible, since the projection is still applied in each trial of the line search. Second, the strategy inherits the iterative nature of the classical method, so the online cost per problem is proportional to the number of iterations, although each iteration is considerably cheaper because it only requires two forward passes of the networks instead of two differential solvers.\\

Concerning the generation of the supervised datasets, each sample of this strategy requires one forward and one backward resolution, that is, the cost of Step 1 of a single iteration of the classical method. In addition, note that the two operators \eqref{eq:operadores_indirecta} do not depend on the algorithm, thus, it is possible to reuse the trained networks with a different descent algorithm or with different parameters of the same one, without generating new data.

\subsection{The direct strategy}\label{sec:estrategia_directa}

The \textit{direct strategy} employs a single DeepONet which approximates the optimal control operator \eqref{eq:general_control_operator},
\[
\mathcal{G}_{c}:[q^{0},\lambda]\longmapsto q^{*},
\]
where $q^{0}$ is the initial control through which the descent algorithm begins, and $q^{*}$ is the optimal control returned by the algorithm. The input to the branch network is now larger than in \eqref{eq:operadores_indirecta}, since the returned control depends on the whole datum $\lambda=[\lambda_e,\lambda_J,\lambda_U]$, including the bound $\lambda_U$ that defines the admissible set and does not enter \eqref{eq:operadores_indirecta}, and on the initial control $q^0$. Denoting the trained network by $\mathcal{G}_{c}^{\vartheta_c}$, a new optimal control is obtained through a single evaluation,
\[
q^{*}\approx\mathcal{G}_{c}^{\vartheta_c}[q^{0},\lambda],
\]
so that the iterative loop of the Section \ref{sec:algoritmo_clasico} disappears completely.\\

\subsection{Offline cost and online evaluation}\label{sec:balance_general}

In both strategies, the offline stage represents the cost of generating the supervised data and training the network only once. This cost is amortized over the online stage, which consists of evaluating the networks each time a new problem of the familiy is solved.\\

Let $T^{\mathrm{off}}$ be the total time of the offline stage, $\tau_{\mathrm{cl}}$ the average time required by the classical method to solve a problem from the family of optimal control problems, and $\tau_{\mathrm{on}}$ the average time of the corresponding online evaluation. The number of problems after which a strategy becomes advantageous in total time is:
\begin{equation}\label{eq:general_balance}
N^{*}=\frac{T^{\mathrm{off}}}{\tau_{\mathrm{cl}}-\tau_{\mathrm{on}}},
\end{equation}
provided that $\tau_{\mathrm{on}}<\tau_{\mathrm{cl}}$; otherwise, the offline stage is never offset. Below the threshold $N^{*}$, the classical method, which does not require a setup stage, is preferable.\\

For the indirect strategy, $T^{\mathrm{off}}$ is the sum of the time required to generate the supervised datasets; each sample of which requires solving one forward and one backward system, and the training times of the state and adjoint networks. On the other hand, $\tau_{\mathrm{on}}$ corresponds to the time for the full descent algorithm with the two solvers replaced by network evaluations.\\

For the direct strategy, $T^{\mathrm{off}}$ is the sum of the time required to generate its own datasets, which requires solving a full optimal control problem per sample, and the training time for the single network, while $\tau_{\mathrm{on}}$ is the time for a single evaluation of that network.

%%%%%%%%%%%%%%%%%%%%%%%%%%%%%%%%%%%%%%%%%%%%%%%%%%%%%%%%%%%%%
% 4. EXAMPLE
%%%%%%%%%%%%%%%%%%%%%%%%%%%%%%%%%%%%%%%%%%%%%%%%%%%%%%%%%%%%%
\section{An application: optimal cytotoxic and antiangiogenic therapies}\label{sec:ejemplo}

In this section we will apply the framework developed in Section \ref{sec:estrategias} to an optimal control problem governed by a reduced model for the growth of glioma cells in the presence of oxygen, in which the controls represent a chemotherapeutic and an antiangiogenic therapy.

\subsection{The state system and the optimal control problem}\label{sec:modelo}

Let $T>0$ be a fixed final time. We consider the state system
\begin{equation}\label{eq:sistema_EDO}
\begin{cases}
u'(t)=\rho(\sigma(t))u(t)(\alpha-u(t))-\kappa\,\ControlC(t)u(t)\sigma(t), & t\in(0,T),\\[1mm]
\sigma'(t)=\gamma(\beta-\sigma(t))-\dfrac{A_{ox}u(t)\sigma(t)}{k_{ox}+\sigma(t)}+(1-\ControlS(t))S_{c}u(t), & t\in(0,T),\\[2mm]
u(0)=u_0,\qquad \sigma(0)=\sigma_0,
\end{cases}
\end{equation}
where $u(t)$ denotes the tumor cell density and $\sigma(t)$ the oxygen concentration, both understood as average quantities over the tissue. The term $\rho(\sigma)u(\alpha-u)$ represents a logistic growth modulated by the oxygen concentration, where $\alpha>0$ is the carrying capacity and $\rho(\sigma)$ is the oxygen dependent proliferation rate; a prototype function used in the literature is
\begin{equation}\label{eq:rho}
\rho(\sigma)=\frac{\widehat{\rho}}{\alpha}\left(\frac{\sigma}{\beta}+b\left(1-\frac{\sigma}{\beta}\right)\right),
\end{equation}
where $\widehat{\rho}$ is the maximal proliferation rate under normoxic conditions and $b\in[0,1]$ modulates the extent to which low oxygen levels reduce proliferation \cite{Gomez2017, Leo2}. In general, $\rho:\mathbb{R}^{+}\to\mathbb{R}$ is assumed to be of class $C^{1}$, nonnegative and with bounded derivative. The term $-\kappa\,\ControlC u\sigma$ describes the decrease of the tumor cells produced by a cytotoxic drug, where $\kappa>0$ measures the maximal cytotoxic therapeutic intensity and the product $u\sigma$ reflects that the efficacy of the drug is modulated by the availability of oxygen. In the oxygen equation, $\gamma:=P_{er}S_{v}>0$ collects the vascular permeability and the vascular density, so that $\gamma(\beta-\sigma)$ represents the supply of oxygen from the vasculature towards the tissue with reference level $\beta>0$; the term $A_{ox}u\sigma/(k_{ox}+\sigma)$, with $A_{ox},k_{ox}>0$, describes the consumption of oxygen by the tumor cells according to Michaelis--Menten kinetics; and $(1-\ControlS)S_{c}u$, with $S_{c}>0$, models the contribution of the angiogenic factors released by the tumor, modulated by the antiangiogenic therapy.\\

The two therapies are represented by the controls $[\ControlC,\ControlS]$, which satisfy $0\leq\ControlC(t),\ControlS(t)\leq1$ a.e. $t\in[0,T]$. The following result guarantees that \eqref{eq:sistema_EDO} is well-posed on the interval $[0,T]$ and that the solution is bounded.

\begin{theo}\label{teo:existencia_edo}
Let $[\ControlC,\ControlS]\in [L^\infty(0,T)]^{2}$ be such that $0\leq\ControlC(t),\ControlS(t)\leq1$ a.e. $t\in[0,T]$. Assume that $0\leq u_0\leq\alpha$ and $0\leq\sigma_0\leq C_\sigma$, with $C_\sigma:=\beta+\frac{S_{c}\alpha}{\gamma}$. Then system \eqref{eq:sistema_EDO} admits a unique absolutely continuous solution $[u,\sigma]$ on $[0,T]$. Moreover, for almost every $t\in[0,T]$,
\[
0\leq u(t)\leq\alpha,\qquad 0\leq\sigma(t)\leq C_\sigma.
\]
\end{theo}

\begin{proof}
Let $[\ControlC, \ControlS]$ be fixed. The right-hand side of \eqref{eq:sistema_EDO} is measurable with respect to $t$ and locally Lipschitz with respect to $[u,\sigma]$, for almost every $t$. Indeed, the controls are measurable and bounded, $\rho$ is of class $C^1$ and hence locally Lipschitz, and the function $\sigma \mapsto \frac{\sigma}{k_{ox}+\sigma}$ is locally Lipschitz for $\sigma > -k_{ox}$. Therefore, by the existence theory under Carathéodory conditions \cite[Chapter 2]{Coddington}, there exists a unique maximal solution $[u,\sigma]$, absolutely continuous on $[0,T_{max})$ for some $T_{max} > 0$ and satisfying \eqref{eq:sistema_EDO} at almost every $t$. We will show that this solution remains bounded on $[0,T_{max})$ and hence can be extended to $[0,T]$.\\

We use a comparison argument. Set $f_u(t,u,\sigma)=\rho(\sigma)u(\alpha-u)-\kappa\ControlC u\sigma$. The constant function $v\equiv0$ satisfies $v'(t)=0=f_u(t,0,\sigma(t))$ and $v(0)=0\leq u_0$, so that $v=0$ is a subsolution of \eqref{eq:sistema_EDO}$_1$ and, by comparison, $u(t)\geq0$. Analogously, setting $f_\sigma(t,u,\sigma)=\gamma(\beta-\sigma)-\frac{A_{ox}u\sigma}{k_{ox}+\sigma}+(1-\ControlS)S_{c}u$ and using $u\geq0$ and $0\leq\ControlS\leq1$, the constant function $v\equiv0$ satisfies $v'(t)=0\leq f_\sigma(t,u(t),0)$ and $v(0)=0\leq\sigma_0$, whence $\sigma(t)\geq0$. Next, since $u\geq0$, $\sigma\geq0$ and $\ControlC\geq0$, from \eqref{eq:sistema_EDO}$_1$ we get
\[
u'(t)=\rho(\sigma(t))u(t)(\alpha-u(t))-\kappa\ControlC(t)u(t)\sigma(t)\leq\rho(\sigma(t))u(t)(\alpha-u(t)),
\]
and the constant function $y\equiv\alpha$ solves $y'=\rho(\sigma(t))y(\alpha-y)$, $y(0)=\alpha$. Since $u_0\leq\alpha$, comparison yields $u(t)\leq\alpha$, and therefore $0\leq u(t)\leq\alpha$. Finally, using $u\leq\alpha$, $0\leq\ControlS\leq1$ and the nonnegativity of $u$ and $\sigma$, from \eqref{eq:sistema_EDO}$_2$ we obtain
\[
\sigma'(t)\leq\gamma(\beta-\sigma(t))+S_{c}\alpha=\gamma\bigl(C_\sigma-\sigma(t)\bigr).
\]
Let $z$ be the solution of $z'=\gamma(C_\sigma-z)$, $z(0)=C_\sigma$, that is, $z\equiv C_\sigma$. Since $\sigma_0\leq C_\sigma$, comparison gives $\sigma(t)\leq C_\sigma$ and hence $0\leq\sigma(t)\leq C_\sigma$ a.e. $t$. These bounds show that the solution remains in the compact set $[0,\alpha] \times [0,C_\sigma]$, on which the right-hand side of \eqref{eq:sistema_EDO} is bounded; hence the maximal solution cannot cease to exist in finite time and $T_{max} > T$, which finishes the proof.
\end{proof}

From a biomedical perspective, the objective is to reduce the tumor density and to regulate the oxygen concentration towards a desired state, while the administered dosages are minimized in order to mitigate the adverse effects of the therapies. Consequently, given nonnegative and bounded weighting functions $k_i\in L^\infty(0,T)$, $i=1,\dots,4$, and constants $l_1,l_2\geq0$, we consider the cost functional
\begin{equation}\label{eq:funcional_costo}
\begin{split}
J([u(\ControlC,\ControlS),\sigma(\ControlC,\ControlS),\ControlC,\ControlS])={}&\int_0^{T}\left[\frac{k_1}{2}u^{2}+\frac{k_2}{2}(\sigma-\sigma_d)^{2}+k_3\ControlC+k_4\ControlS\right](t)\,dt\\
&+\frac{l_1}{2}u(T)^{2}+\frac{l_2}{2}(\sigma(T)-\sigma_{dT})^{2},
\end{split}
\end{equation}
where $\sigma_d\in L^\infty(0,T)$ represents the desired oxygen state along the whole interval $(0,T)$ and $\sigma_{dT}\in\mathbb{R}$ the desired oxygen value at the final time $T$. The terms $\frac{k_1}{2}u^{2}$ and $\frac{l_1}{2}u(T)^2$ reduce the tumor cell density during the treatment and at the end of the therapeutic time, respectively; the terms $\frac{k_2}{2}(\sigma-\sigma_d)^{2}$ and $\frac{l_2}{2}(\sigma(T)-\sigma_{dT})^{2}$ express the interest in keeping the oxygen close to a reference level, avoiding both hypoxia and an excessive vascularization; and the terms $k_3\ControlC$ and $k_4\ControlS$ account for the amount of administered drug.\\

The set of admissible controls is defined by
\begin{equation}\label{eq:conjunto_admisible}
\mathcal{U}_{ad}:=\left\{[\ControlC,\ControlS]\in [L^\infty(0,T)]^{2}:\ 0\leq\ControlC(t),\ControlS(t)\leq1\ \text{ a.e. in }(0,T),\ \int_0^{T}\ControlC(t)\,dt\leq c_{\max}\right\},
\end{equation}
where $0<c_{\max}<T$ denotes the maximal total amount of cytotoxic therapy that can realistically be administered with standard drugs. The pointwise constraints represent that the intensity of a therapy can neither exceed a physiologically admissible maximum nor be negative, whereas the integral constraint on $\ControlC$ represents the limit of the accumulated dose of the cytotoxic drug. Since the states depend on the controls, we consider the reduced functional
\[
\widehat{J}([\ControlC,\ControlS]):=J([u(\ControlC,\ControlS),\sigma(\ControlC,\ControlS),\ControlC,\ControlS]),
\]
and thus, the optimization problem to be analyzed is
\begin{equation}\label{eq:problema_control}
(P)\quad\left\{
\begin{array}{ll}
\min & \widehat{J}([\ControlC,\ControlS])\\[2mm]
\text{subject to} & \left\{\begin{array}{l}
[u(\ControlC,\ControlS),\sigma(\ControlC,\ControlS)]\ \text{is the solution of the state system \eqref{eq:sistema_EDO}}\\
\text{associated with the admissible controls }[\ControlC,\ControlS]\in\mathcal{U}_{ad}.
\end{array}\right.
\end{array}\right.
\end{equation}
A pair $[\bar \ControlC, \bar \ControlS] \in \mathcal{U}_{ad}$ is called a \textit{global optimal control} of (P) when it minimizes $\widehat{J}$ on the whole set $\mathcal{U}_{ad}$, and a \textit{local optimal control} when it minimizes $\widehat{J}$ only on a neighborhood of $[\bar \ControlC, \bar \ControlS]$ in $\mathcal{U}_{ad}$. The following result, whose proof follows a standard minimizing sequence argument, guarantees that the optimal control problem is solvable.

\begin{theo}\label{teo:existencia_control}
Assume that $0\leq u_0\leq\alpha$ and $0\leq\sigma_0\leq C_\sigma$. Then problem \eqref{eq:problema_control} admits at least one global optimal control.
\end{theo}

Problem \eqref{eq:problema_control} fits into the framework of Section \ref{sec:problema_general} with
\begin{equation}\label{eq:identificacion_lambda}
q=[\ControlC,\ControlS],\qquad \lambda_e=[u_0,\sigma_0],\qquad \lambda_J=[k_1,k_2,k_3,k_4,l_1,l_2],\qquad \lambda_U=[c_{\max}],
\end{equation}
the remaining parameters of the model being kept fixed, the space of controls being $\mathcal{Q}=[L^2(0,T)]^2$ and the spaces of states and adjoint being $\mathcal{Y} = \mathcal{Z} = [C([0,T])]^2$. Hypothesis \textbf{(H1)} is then given by Theorem \ref{teo:existencia_edo} and \textbf{(H3)} by Theorem \ref{teo:existencia_control}, whereas \textbf{(H2)} and \textbf{(H4)} are the usual differentiability and adjoint representation properties for this class of problems \cite{Hinze, Troltzsch}; since we will follow a {\it discretize-then-optimize approach}, both are verified directly at the discrete level in Section \ref{sec:esquema}. The weights $k_i$ and $l_j$ represent the clinical decision about what is privileged in each case, and different choices give rise to different optimal treatments; this is precisely the source of the parametric family that motivates the strategies of Section \ref{sec:estrategias}.

\subsection{A discretize-then-optimize scheme}\label{sec:esquema}
In this subsection we describe the discretization on which the classical method of Section \ref{sec:algoritmo_clasico} is built for this example. We follow a \textit{discretize-then-optimize} approach, this is, the state system and the cost functional are discretized first, and the sensitivity system, the adjoint system, the reduced functional and its gradient are then derived at the discrete level. In this way, the gradient that is used in the descent method is the exact gradient of the discrete cost functional, and no inconsistency between the discrete functional and its computed gradient is introduced.\\

We consider a uniform partition of $[0,T]$ with time step $h=T/N$ and nodes $t_n=nh$, $n=0,\dots,N$, and we write $u^{n}\approx u(t_n)$, $\sigma^{n}\approx\sigma(t_n)$, $k_i^{n}=k_i(t_n)$ and $\sigma_d^{n}=\sigma_d(t_n)$. The controls are taken piecewise constant, with value $\ControlC^{n}$ and $\ControlS^{n}$ on the interval $(t_{n-1},t_n]$, $n=1,\dots,N$, so that a pair of discrete controls is identified with a vector
\[
[\ControlC,\ControlS]=\bigl[(\ControlC^{1},\dots,\ControlC^{N}),(\ControlS^{1},\dots,\ControlS^{N})\bigr]\in\mathbb{R}^{N}\times\mathbb{R}^{N}.
\]
For $[\ControlC, \ControlS]$ and $[\bar \ControlC, \bar \ControlS]$ in $\mathbb{R}^N \times \mathbb{R}^N$, we introduce the discrete $L^2$ inner product and its induced norm
\[
([\ControlC,\ControlS], [\bar \ControlC, \bar \ControlS])_h := h \sum_{j=1}^N (\ControlC^j \bar \ControlC^j + \ControlS^j\bar \ControlS^j), \qquad \|[\ControlC, \ControlS]\|_h^2 := h\sum_{j=1}^N ((\ControlC^j)^2 + (\ControlS^j)^2).
\]

\subsubsection*{Discretization of the cost functional}

In order to discretize the cost functional \eqref{eq:funcional_costo} we use the trapezoidal rule
$\int_a^{b}f\approx \Delta x\bigl[\frac{f(x_0)}{2}+\sum_{i=1}^{M-1}f(x_i)+\frac{f(x_M)}{2}\bigr]$, with $\Delta x = (b-a)/M$ and $x_i = a + i \Delta x$, for the quadratic terms. Thus, the discrete cost functional is given by
\begin{equation}\label{eq:funcional_discreto}
\begin{split}
J_h(u,\sigma,\ControlC,\ControlS)={}&h\sum_{n=1}^{N-1}\Bigl[\frac{k_1^{n}}{2}(u^{n})^{2}+\frac{k_2^{n}}{2}(\sigma^{n}-\sigma_d^{n})^{2}\Bigr]
+\frac{h}{2}\Bigl[\frac{k_1^{0}}{2}(u^{0})^{2}+\frac{k_2^{0}}{2}(\sigma^{0}-\sigma_d^{0})^{2}\Bigr]\\
&+\frac{h}{2}\Bigl[\frac{k_1^{N}}{2}(u^{N})^{2}+\frac{k_2^{N}}{2}(\sigma^{N}-\sigma_d^{N})^{2}\Bigr]
+h\sum_{n=1}^{N}\bigl[k_3^{n}\ControlC^{n}+k_4^{n}\ControlS^{n}\bigr]\\
&+\frac{l_1}{2}(u^{N})^{2}+\frac{l_2}{2}(\sigma^{N}-\sigma_{dT})^{2},
\end{split}
\end{equation}
where the control terms keep the indexation $n=1,\dots,N$ of the numerical scheme, since $\ControlC$ and $\ControlS$ are constant on each interval $(t_{n-1},t_n]$ and their integral does not depend on the chosen quadrature rule.

\subsubsection*{Numerical scheme for the state variables}

\noindent
\textbf{Initialization:} Let $u^{0}=u_0$ and $\sigma^{0}=\sigma_0$ with $0\leq u_0\leq\alpha$ and $0\leq\sigma_0\leq\beta+\frac{S_{c}\alpha}{\gamma}$.\\

\noindent
\textbf{Time step $n$:} Given $u^{n-1}$ and $\sigma^{n-1}$, compute $u^{n}$ and $\sigma^{n}$ from
\begin{equation}\label{eq:esquema_estados}
\begin{cases}
\dfrac{u^{n}-u^{n-1}}{h}=\rho(\sigma^{n-1})u^{n-1}(\alpha-u^{n})-\kappa\,\ControlC^{n}u^{n}\sigma^{n-1},\\[3mm]
\dfrac{\sigma^{n}-\sigma^{n-1}}{h}=\gamma(\beta-\sigma^{n})-\dfrac{A_{ox}u^{n}\sigma^{n}}{k_{ox}+\sigma^{n-1}}+(1-\ControlS^{n})S_{c}u^{n}.
\end{cases}
\end{equation}

\begin{prop}\label{prop:esquema_estados}
For every $h>0$ and every pair of discrete controls with $\ControlC^{n}\geq0$ and $0 \leq \ControlS^{n}\leq1$, scheme \eqref{eq:esquema_estados} determines $u^{n}$ and $\sigma^{n}$ in a unique way. Moreover,
\[
0\leq u^{n}\leq\alpha,\qquad 0\leq\sigma^{n}\leq\beta+\frac{S_{c}\alpha}{\gamma},\qquad n=0,\dots,N.
\]
\end{prop}

\begin{proof}
Let $n \geq 0$ be fixed and assume (induction hypothesis) $0\leq u^{n-1}\leq\alpha$ and $0\leq\sigma^{n-1}\leq\beta+\frac{S_{c}\alpha}{\gamma}$. Equation \eqref{eq:esquema_estados}$_1$ can be written as
\begin{equation}\label{eq:un_explicito}
u^{n}=\frac{u^{n-1}\bigl(1+h\alpha\rho(\sigma^{n-1})\bigr)}{1+h\rho(\sigma^{n-1})u^{n-1}+h\kappa \ControlC^{n}\sigma^{n-1}}.
\end{equation}
The denominator is greater than or equal to one, hence there exists a unique $u^{n}$. Moreover, since $\ControlC^{n}\geq0$ and $\rho(\sigma^{n-1})\geq0$, we get $u^{n}\geq0$.  For the upper bound, adding $\alpha-u^{n}$ on both sides of \eqref{eq:esquema_estados}$_1$, gives
\[
\bigl(1+h\rho(\sigma^{n-1})u^{n-1}\bigr)(\alpha-u^{n})=(\alpha-u^{n-1})+h\kappa \ControlC^{n}u^{n}\sigma^{n-1}\geq\alpha-u^{n-1}\geq0,
\]
where we used that $u^{n-1} \leq \alpha$. As $1+h\rho(\sigma^{n-1})u^{n-1}>0$, this yields $u^{n}\leq\alpha$. Once $u^{n}$ has been computed, equation \eqref{eq:esquema_estados}$_2$ is equivalent to
\begin{equation}\label{eq:sigman_explicito}
\sigma^{n}=\frac{\sigma^{n-1}+h\gamma\beta+h(1-\ControlS^{n})S_{c}u^{n}}{1+h\gamma+h\frac{A_{ox}u^{n}}{k_{ox}+\sigma^{n-1}}}.
\end{equation}
Analogously, the denominator is greater than or equal to one, so that there exists a unique $\sigma^{n}$; and since $\ControlS^{n}\leq1$ and $u^{n}\geq0$, the numerator is nonnegative and hence $\sigma^{n}\geq0$. Now discarding the nonpositive term $-h\frac{A_{ox}u^{n}\sigma^{n}}{k_{ox}+\sigma^{n-1}}$ and using that $(1-\ControlS^{n})u^{n}\leq u^{n}\leq\alpha$ on \eqref{eq:esquema_estados}$_2$, we get $\sigma^{n}-\sigma^{n-1}\leq h\gamma(\beta-\sigma^{n})+hS_{c}\alpha$. Adding $\bigl(\beta+\frac{S_{c}\alpha}{\gamma}\bigr)-\sigma^{n}$ in both sides, this becomes
\[
(1+h\gamma)\Bigl[\Bigl(\beta+\frac{S_{c}\alpha}{\gamma}\Bigr)-\sigma^{n}\Bigr]\geq\Bigl(\beta+\frac{S_{c}\alpha}{\gamma}\Bigr)-\sigma^{n-1}\geq0.
\]
As $1+h\gamma>0$, this yields $\sigma^{n}\leq\beta+\frac{S_{c}\alpha}{\gamma}$, which concludes the induction step and, since the bounds hold for $n=0$, the proof.
\end{proof}

\subsubsection*{The discrete sensitivity problem}
By \eqref{eq:un_explicito} and \eqref{eq:sigman_explicito}, each $u^{n}$ and each $\sigma^{n}$ are obtained from the discrete controls by a finite composition of rational expressions whose denominators are greater than or equal to one when $0\leq\ControlC^{n},\ControlS^{n}\leq1$, in addition, since $\rho$ is of class $C^{1}$, the discrete states, and therefore also the discrete cost functional \eqref{eq:funcional_discreto}, are continuously differentiable functions of the discrete controls in an open neighborhood of $[0,1]^{2N}$.\\

Fix $j\in\{1,\dots,N\}$ and consider the sensitivities of the state variables with respect to the components $\ControlC^{j}$ and $\ControlS^{j}$,
\[
U_c^{n}=\frac{\partial u^{n}}{\partial \ControlC^{j}},\quad V_c^{n}=\frac{\partial\sigma^{n}}{\partial \ControlC^{j}},\quad U_s^{n}=\frac{\partial u^{n}}{\partial \ControlS^{j}},\quad V_s^{n}=\frac{\partial\sigma^{n}}{\partial \ControlS^{j}},
\]
for $n\geq j$. Note that $u^n$ and $\sigma^n$ depend only on the values of the controls on the first $n$ subintervals, so that all these sensitivities vanish for $n < j$. Taking into account that
\[
\frac{\partial \ControlC^{n}}{\partial \ControlC^{j}}=\delta_{nj},\qquad\frac{\partial \ControlS^{n}}{\partial \ControlS^{j}}=\delta_{nj},
\]
where $\delta_{nj}$ denotes the Kronecker delta, and differentiating \eqref{eq:esquema_estados}, we obtain the following scheme for the sensitivity problem with respect to $\ControlC$ and $\ControlS$.\\

\noindent
\textbf{Initialization:} Let $U_c^{0}=0$ and $V_c^{0}=0$.\\

\noindent
\textbf{Time step $n=1,\dots,N$:} Given $U_c^{n-1}$ and $V_c^{n-1}$, compute $U_c^{n}$ and $V_c^{n}$ from
\begin{equation}\label{eq:sensibilidad_c}
\begin{cases}
\begin{aligned}
\dfrac{U_c^{n}-U_c^{n-1}}{h}={}&\rho'(\sigma^{n-1})u^{n-1}(\alpha-u^{n})V_c^{n-1}+\rho(\sigma^{n-1})(\alpha-u^{n})U_c^{n-1}-\rho(\sigma^{n-1})u^{n-1}U_c^{n}\\
&-\kappa \ControlC^{n}u^{n}V_c^{n-1}-\kappa \ControlC^{n}\sigma^{n-1}U_c^{n}-\kappa u^{n}\sigma^{n-1}\delta_{nj},
\end{aligned}\\[6mm]
\begin{aligned}
\dfrac{V_c^{n}-V_c^{n-1}}{h}={}&-\gamma V_c^{n}-\frac{A_{ox}\sigma^{n}}{k_{ox}+\sigma^{n-1}}U_c^{n}-\frac{A_{ox}u^{n}}{k_{ox}+\sigma^{n-1}}V_c^{n}\\
&+\frac{A_{ox}u^{n}\sigma^{n}}{(k_{ox}+\sigma^{n-1})^{2}}V_c^{n-1}+(1-\ControlS^{n})S_{c}U_c^{n}.
\end{aligned}
\end{cases}
\end{equation}

Analogously, we obtain the sensitivity problem with respect to $\ControlS$:\\

\noindent
\textbf{Initialization:} Let $U_s^{0}=0$ and $V_s^{0}=0$.\\

\noindent
\textbf{Time step $n=1,\dots,N$:} Given $U_s^{n-1}$ and $V_s^{n-1}$, compute $U_s^{n}$ and $V_s^{n}$ from
\begin{equation}\label{eq:sensibilidad_s}
\begin{cases}
\begin{aligned}
\dfrac{U_s^{n}-U_s^{n-1}}{h}={}&\rho'(\sigma^{n-1})u^{n-1}(\alpha-u^{n})V_s^{n-1}+\rho(\sigma^{n-1})(\alpha-u^{n})U_s^{n-1}-\rho(\sigma^{n-1})u^{n-1}U_s^{n}\\
&-\kappa \ControlC^{n}u^{n}V_s^{n-1}-\kappa \ControlC^{n}\sigma^{n-1}U_s^{n},
\end{aligned}\\[6mm]
\begin{aligned}
\dfrac{V_s^{n}-V_s^{n-1}}{h}={}&-\gamma V_s^{n}-\frac{A_{ox}\sigma^{n}}{k_{ox}+\sigma^{n-1}}U_s^{n}-\frac{A_{ox}u^{n}}{k_{ox}+\sigma^{n-1}}V_s^{n}\\
&+\frac{A_{ox}u^{n}\sigma^{n}}{(k_{ox}+\sigma^{n-1})^{2}}V_s^{n-1}+(1-\ControlS^{n})S_{c}U_s^{n}-S_{c}u^{n}\delta_{nj}.
\end{aligned}
\end{cases}
\end{equation}

Note that both sensitivity systems have the same structure as the state scheme \eqref{eq:esquema_estados}; in fact, $U^{n}$ and $V^{n}$ are obtained successively from the same denominators as in \eqref{eq:un_explicito} and \eqref{eq:sigman_explicito}, which are greater than or equal to one, so that both systems are uniquely solvable and explicitly computable. However, solving them for each $j=1,\dots,N$ would require $2N$ additional solutions per iteration of the descent method. This is precisely what the discrete adjoint system, introduced below, allows us to avoid.

\subsubsection*{The discrete adjoint system}

\noindent
\textbf{Initialization:} Compute $z^{N}$ and $w^{N}$ from
\begin{equation}\label{eq:adjunto_terminal}
\begin{split}
&\Bigl(1+h\gamma+h\frac{A_{ox}u^{N}}{k_{ox}+\sigma^{N-1}}\Bigr)z^{N}=\frac{h}{2}k_2^{N}(\sigma^{N}-\sigma_d^{N})+l_2(\sigma^{N}-\sigma_{dT}),\\
&\Bigl(1+h\rho(\sigma^{N-1})u^{N-1}+h\kappa \ControlC^{N}\sigma^{N-1}\Bigr)w^{N}+h\Bigl(\frac{A_{ox}\sigma^{N}}{k_{ox}+\sigma^{N-1}}-(1-\ControlS^{N})S_{c}\Bigr)z^{N}=\frac{h}{2}k_1^{N}u^{N}+l_1u^{N}.
\end{split}
\end{equation}

\noindent
\textbf{Time step $n=N-1,\dots,1$:} Given $z^{n+1}$ and $w^{n+1}$, compute $z^{n}$ and $w^{n}$ from
\begin{equation}\label{eq:adjunto}
\begin{cases}
\begin{aligned}
\dfrac{z^{n}-z^{n+1}}{h}={}&\rho'(\sigma^{n})u^{n}(\alpha-u^{n+1})w^{n+1}-\kappa \ControlC^{n+1}u^{n+1}w^{n+1}-\gamma z^{n}-\frac{A_{ox}u^{n}}{k_{ox}+\sigma^{n-1}}z^{n}\\
&+\frac{A_{ox}u^{n+1}\sigma^{n+1}}{(k_{ox}+\sigma^{n})^{2}}z^{n+1}+k_2^{n}(\sigma^{n}-\sigma_d^{n}),
\end{aligned}\\[6mm]
\begin{aligned}
\dfrac{w^{n}-w^{n+1}}{h}={}&\rho(\sigma^{n})(\alpha-u^{n+1})w^{n+1}-\rho(\sigma^{n-1})u^{n-1}w^{n}-\kappa \ControlC^{n}\sigma^{n-1}w^{n}\\
&-\frac{A_{ox}\sigma^{n}}{k_{ox}+\sigma^{n-1}}z^{n}+(1-\ControlS^{n})S_{c}z^{n}+k_1^{n}u^{n}.
\end{aligned}
\end{cases}
\end{equation}
Here, $w^{n}$ and $z^{n}$ are the discrete adjoint variables associated with the tumor cell density and oxygen concentration, respectively. As in the continuous case, the adjoint system is solved backward in time. Note that it has a unique solution and, like the state scheme, is decoupled, since in this case, at each time step, $z^{n}$ is obtained first and then $w^{n}$.

\subsubsection*{The discrete gradient}

The discrete counterpart of the admissible set \eqref{eq:conjunto_admisible} is
\begin{equation}\label{eq:admisible_discreto}
\mathcal{U}_{ad}^{h}:=\Bigl\{[\ControlC,\ControlS]\in\mathbb{R}^{N}\times\mathbb{R}^{N}:\ 0\leq \ControlC^{j}\leq1,\ 0\leq \ControlS^{j}\leq1,\ j=1,\dots,N,\ \ h\sum_{j=1}^{N}\ControlC^{j}\leq c_{\max}\Bigr\},
\end{equation}
which is a nonempty, convex and compact subset of $\mathbb{R}^{2N}$. It is convenient to write this set as a Cartesian product
\begin{equation}\label{eq:cartesianProduct}
  \mathcal{U}_{ad}^h = \mathcal{K}_\ControlC^h \times \mathcal{K}_\ControlS^h, \quad \mathcal{K}^h := [0,1]^N, \quad \mathcal{K}_\ControlC^h := \big\{v\in \mathcal{K}^h : h \sum_{j=1}^N v^j \leq c_{\text{max}}\big\}, \quad \mathcal{K}_\ControlS^h := \mathcal{K}^h.
\end{equation}
Consequently, defining the discrete reduced functional by
\[
\widehat{J}_h(\ControlC,\ControlS) := J_h\bigl({u(\ControlC,\ControlS),\sigma(\ControlC,\ControlS),\ControlC,\ControlS}\bigr),
\]
we consider the discrete problem
\begin{equation}\label{eq:problema_discreto}
(P_h)\qquad\min_{{[\ControlC,\ControlS]}\in\mathcal{U}_{ad}^{h}}\widehat{J}_h\bigl(\ControlC,\ControlS\bigr).
\end{equation}
Since $\mathcal{U}_{ad}^h$ is a nonempty compact set and $\widehat{J}_h$ is continuous over this set, then by the Weierstrass Theorem, the problem $(P_h)$ admits at least one global solution,. Applying the chain rule in \eqref{eq:funcional_discreto}, and taking into account that $u^{0}$ and $\sigma^{0}$ do not depend on the controls, we obtain
\begin{equation}\label{eq:derivada_cadena_c}
\begin{split}
\frac{\partial \widehat{J}_h}{\partial \ControlC^{j}}={}&h\left[\sum_{n=1}^{N-1}k_1^{n}u^{n}U_c^{n}+\sum_{n=1}^{N-1}k_2^{n}(\sigma^{n}-\sigma_d^{n})V_c^{n}+k_3^{j}\right]\\
&+\frac{h}{2}\Bigl[k_1^{N}u^{N}U_c^{N}+k_2^{N}(\sigma^{N}-\sigma_d^{N})V_c^{N}\Bigr]+l_1u^{N}U_c^{N}+l_2(\sigma^{N}-\sigma_{dT})V_c^{N},
\end{split}
\end{equation}
and, analogously,
\begin{equation}\label{eq:derivada_cadena_s}
\begin{split}
\frac{\partial \widehat{J}_h}{\partial \ControlS^{j}}={}&h\left[\sum_{n=1}^{N-1}k_1^{n}u^{n}U_s^{n}+\sum_{n=1}^{N-1}k_2^{n}(\sigma^{n}-\sigma_d^{n})V_s^{n}+k_4^{j}\right]\\
&+\frac{h}{2}\Bigl[k_1^{N}u^{N}U_s^{N}+k_2^{N}(\sigma^{N}-\sigma_d^{N})V_s^{N}\Bigr]+l_1u^{N}U_s^{N}+l_2(\sigma^{N}-\sigma_{dT})V_s^{N}.
\end{split}
\end{equation}
The following result shows that the adjoint system \eqref{eq:adjunto_terminal}--\eqref{eq:adjunto} eliminates the sensitivities from \eqref{eq:derivada_cadena_c} and \eqref{eq:derivada_cadena_s}, and provides a closed expression for the discrete reduced gradient.

\begin{prop}\label{prop:gradiente_discreto}
Let $[u,\sigma]$ be the solution of the state scheme \eqref{eq:esquema_estados} associated with the discrete controls $[\ControlC,\ControlS]$, and let $[w,z]$ be the solution of the adjoint system \eqref{eq:adjunto_terminal}--\eqref{eq:adjunto}. Then, for every $j=1,\dots,N$,
\begin{equation}\label{eq:gradiente_discreto}
\frac{\partial \widehat{J}_h}{\partial \ControlC^{j}}=h\bigl(k_3^{j}-\kappa\,u^{j}\sigma^{j-1}w^{j}\bigr),\qquad
\frac{\partial \widehat{J}_h}{\partial \ControlS^{j}}=h\bigl(k_4^{j}-S_{c}\,u^{j}z^{j}\bigr).
\end{equation}
Consequently, the gradient of $\widehat{J}_h$ with respect to the discrete $L^2$ inner product defined above is given componentwise by
\begin{equation}\label{eq:derivada_riesz}
(\nabla_h \widehat{J}_h(\ControlC,\ControlS))^j_c = k_3^j-\kappa u^j\sigma^{j-1}w^j, \qquad (\nabla_h \widehat{J}_h(\ControlC,\ControlS))^j_s = k_4^j - S_c u^j z^j.
\end{equation}
\end{prop}
\begin{proof}
Multiplying equation \eqref{eq:sensibilidad_c}$_1$ by $hw^{n}$ and equation \eqref{eq:sensibilidad_c}$_2$ by $hz^{n}$, and adding for $n=1,\dots,N$, we obtain
\begin{equation}\label{eq:suma_partes}
\begin{split}
\sum_{n=1}^{N}w^{n}(U_c^{n}-U_c^{n-1})+\sum_{n=1}^{N}z^{n}(V_c^{n}-V_c^{n-1})
={}&h\sum_{n=1}^{N}\Bigl(\rho'(\sigma^{n-1})u^{n-1}(\alpha-u^{n})V_c^{n-1}w^{n}\\
&+\rho(\sigma^{n-1})(\alpha-u^{n})U_c^{n-1}w^{n}-\rho(\sigma^{n-1})u^{n-1}U_c^{n}w^{n}\\
&-\kappa \ControlC^{n}u^{n}V_c^{n-1}w^{n}-\kappa \ControlC^{n}\sigma^{n-1}U_c^{n}w^{n}-\kappa u^{n}\sigma^{n-1}\delta_{nj}w^{n}\\
&-\gamma V_c^{n}z^{n}-\frac{A_{ox}u^{n}}{k_{ox}+\sigma^{n-1}}V_c^{n}z^{n}-\frac{A_{ox}\sigma^{n}}{k_{ox}+\sigma^{n-1}}U_c^{n}z^{n}\\
&+\frac{A_{ox}u^{n}\sigma^{n}}{(k_{ox}+\sigma^{n-1})^{2}}V_c^{n-1}z^{n}+(1-\ControlS^{n})S_{c}U_c^{n}z^{n}\Bigr).
\end{split}
\end{equation}
Since $U_c^{0}=0$ and $V_c^{0}=0$, a summation by parts gives
\[
\sum_{n=1}^{N}w^{n}(U_c^{n}-U_c^{n-1})=\sum_{n=1}^{N-1}(w^{n}-w^{n+1})U_c^{n}+w^{N}U_c^{N},\qquad
\sum_{n=1}^{N}z^{n}(V_c^{n}-V_c^{n-1})=\sum_{n=1}^{N-1}(z^{n}-z^{n+1})V_c^{n}+z^{N}V_c^{N}.
\]
Moreover, the terms of the right-hand side of \eqref{eq:suma_partes} containing $U_c^{n-1}$ and $V_c^{n-1}$ can be reindexed, using again $U_c^{0}=V_c^{0}=0$, as
\[
\sum_{n=1}^{N}h\rho(\sigma^{n-1})(\alpha-u^{n})U_c^{n-1}w^{n}=\sum_{n=1}^{N-1}h\rho(\sigma^{n})(\alpha-u^{n+1})U_c^{n}w^{n+1},
\]
\[
\sum_{n=1}^{N}h\bigl[\rho'(\sigma^{n-1})u^{n-1}(\alpha-u^{n})-\kappa \ControlC^{n}u^{n}\bigr]V_c^{n-1}w^{n}=\sum_{n=1}^{N-1}h\bigl[\rho'(\sigma^{n})u^{n}(\alpha-u^{n+1})-\kappa \ControlC^{n+1}u^{n+1}\bigr]V_c^{n}w^{n+1},
\]
\[
\sum_{n=1}^{N}h\frac{A_{ox}u^{n}\sigma^{n}}{(k_{ox}+\sigma^{n-1})^{2}}V_c^{n-1}z^{n}=\sum_{n=1}^{N-1}h\frac{A_{ox}u^{n+1}\sigma^{n+1}}{(k_{ox}+\sigma^{n})^{2}}V_c^{n}z^{n+1},
\]
while, by the definition of the Kronecker delta,
\[
\sum_{n=1}^{N}h\kappa u^{n}\sigma^{n-1}\delta_{nj}w^{n}=h\kappa u^{j}\sigma^{j-1}w^{j}.
\]
Collecting now, in \eqref{eq:suma_partes}, the coefficients of $U_c^{n}$ and of $V_c^{n}$ for $n=1,\dots,N-1$, and the coefficients of $U_c^{N}$ and $V_c^{N}$, and substituting the adjoint equations \eqref{eq:adjunto} and the terminal conditions \eqref{eq:adjunto_terminal}, we obtain
\begin{equation}\label{eq:identidad_adjunta_c}
\begin{split}
0={}&h\sum_{n=1}^{N-1}k_1^{n}u^{n}U_c^{n}+h\sum_{n=1}^{N-1}k_2^{n}(\sigma^{n}-\sigma_d^{n})V_c^{n}+\frac{h}{2}k_1^{N}u^{N}U_c^{N}+\frac{h}{2}k_2^{N}(\sigma^{N}-\sigma_d^{N})V_c^{N}\\
&+l_1u^{N}U_c^{N}+l_2(\sigma^{N}-\sigma_{dT})V_c^{N}+h\kappa u^{j}\sigma^{j-1}w^{j}.
\end{split}
\end{equation}
Consequently, all the terms of \eqref{eq:derivada_cadena_c} which contain sensitivities add up to $-h\kappa u^{j}\sigma^{j-1}w^{j}$, and substituting this identity into \eqref{eq:derivada_cadena_c} we get the first equation in \eqref{eq:gradiente_discreto}.\\

For the sensitivity with respect to $\ControlS^{j}$ the argument is identical, the only difference being that the term containing the Kronecker delta now appears in the second equation of \eqref{eq:sensibilidad_s} and satisfies
\[
\sum_{n=1}^{N}hS_{c}u^{n}\delta_{nj}z^{n}=hS_{c}u^{j}z^{j}.
\]
Repeating the previous procedure we obtain
\[
\begin{split}
0={}&h\sum_{n=1}^{N-1}k_1^{n}u^{n}U_s^{n}+h\sum_{n=1}^{N-1}k_2^{n}(\sigma^{n}-\sigma_d^{n})V_s^{n}+\frac{h}{2}k_1^{N}u^{N}U_s^{N}+\frac{h}{2}k_2^{N}(\sigma^{N}-\sigma_d^{N})V_s^{N}\\
&+l_1u^{N}U_s^{N}+l_2(\sigma^{N}-\sigma_{dT})V_s^{N}+hS_{c}u^{j}z^{j},
\end{split}
\]
and substituting into \eqref{eq:derivada_cadena_s} yields the second equation in \eqref{eq:gradiente_discreto}. Finally, by the Riesz representation with respect to $(\cdot, \cdot)_h$, the equations in \eqref{eq:gradiente_discreto} give the expression in \eqref{eq:derivada_riesz} for $\nabla_h \widehat{J}_h$.
\end{proof}

\subsubsection*{The discrete first-order necessary optimality conditions}\label{sec:kkt_discreto}
To obtain the necessary first-order conditions for the discrete problem \eqref{eq:problema_discreto}, we proceed in three parts: first, we derive the discrete variational inequality; next, we formulate it component by component; and finally, we incorporate the integral constraint using a Lagrange multiplier.\\

\subsubsection*{Part 1: discrete variational inequality}

\begin{lemm}[Discrete Variational Inequality]\label{lema:vi_discreta}
Let $[\overline{\ControlC},\overline{\ControlS}]\in\mathcal{U}_{ad}^{h}$ be a local optimal control of $(P_h)$. Then
\begin{equation}\label{eq:vi_discreta}
\left(\nabla_h\widehat{J}_h(\overline{\ControlC},\overline{\ControlS}),\,
[\ControlC,\ControlS]-[\overline{\ControlC},\overline{\ControlS}]\right)_h\ \ge\ 0
\qquad\text{for all }[\ControlC,\ControlS]\in\mathcal{U}_{ad}^{h},
\end{equation}
that is, by defining the two components of $\nabla_h\widehat J_h(\overline{\ControlC},\overline{\ControlS})$ given by \eqref{eq:derivada_riesz} as
\begin{equation}\label{eq:gradiente_gc_gs}
g_\ControlC^{\,j}:=k_3^{j}-\kappa\,\overline{u}^{\,j}\overline{\sigma}^{\,j-1}\overline{w}^{\,j},
\qquad
g_\ControlS^{\,j}:=k_4^{j}-S_{c}\,\overline{u}^{\,j}\overline{z}^{\,j},
\qquad j=1,\dots,N,
\end{equation}
then, the variational inequality is obtained
\begin{equation}\label{eq:vi_discreta2}
h\sum_{j=1}^{N}g_\ControlC^{\,j}\bigl(\ControlC^{j}-\overline{\ControlC}^{\,j}\bigr)
+h\sum_{j=1}^{N}g_\ControlS^{\,j}\bigl(\ControlS^{j}-\overline{\ControlS}^{\,j}\bigr)\ \ge\ 0
\qquad\text{for all }[\ControlC,\ControlS]\in\mathcal{U}_{ad}^{h},
\end{equation}
where $[\overline u,\overline\sigma]$ and $[\overline w,\overline z]$ are the discrete states and adjoint states associated with $[\overline{\ControlC},\overline{\ControlS}]$. Furthermore, \eqref{eq:vi_discreta2} can be split into two independent inequalities
\begin{equation}\label{eq:vi_split}
h\sum_{j=1}^{N}g_\ControlC^{\,j}\bigl(v^{j}-\overline{\ControlC}^{\,j}\bigr)\ge0\quad\forall\,v\in\mathcal{K}_{\ControlC}^{h},
\qquad\text{and}\qquad
h\sum_{j=1}^{N}g_\ControlS^{\,j}\bigl(v^{j}-\overline{\ControlS}^{\,j}\bigr)\ge0\quad\forall\,v\in\mathcal{K}_{\ControlS}^{h}.
\end{equation}
\end{lemm}

\begin{proof}
Let $[\ControlC,\ControlS]\in\mathcal{U}_{ad}^{h}$, and let us define
$q_\theta := [\overline{\ControlC}, \overline{\ControlS}] + \theta \bigl([\ControlC, \ControlS] - [\overline{\ControlC}, \overline{\ControlS}]\bigr)$
for $\theta \in (0, 1]$. In fact, note that $q_\theta \in \mathcal{U}_{ad}^h$ and $q_\theta \to [\overline{\ControlC}, \overline{\ControlS}]$ as $\theta \to 0^+$. Furthermore, since $[\overline{\ControlC}, \overline{\ControlS}]$ is a local optimum, we have $\widehat J_h(q_\theta)\ge\widehat J_h(\overline{\ControlC},\overline{\ControlS})$. Since $\widehat J_h$ is of class $C^{1}$ in an open neighborhood of $[0,1]^{2N}$, dividing by $\theta > 0$ and letting $\theta \to 0^{+}$ yields
\begin{equation}\label{eq:teoDeisgV}
0 \le \lim_{\theta \to 0^{+}} \frac{\widehat J_h(q_\theta) - \widehat J_h(\overline{\ControlC}, \overline{\ControlS})}{\theta}
=\bigl(\nabla_h\widehat J_h(\overline{\ControlC},\overline{\ControlS}),[\ControlC,\ControlS]-[\overline{\ControlC},\overline{\ControlS}]\bigr)_h,
\end{equation}
Therefore, \eqref{eq:vi_discreta} is proven. The inequality \eqref{eq:vi_discreta2} follows from \eqref{eq:derivada_riesz} applied to \eqref{eq:teoDeisgV}. Finally, since $\mathcal{U}_{ad}^{h}$ is the Cartesian product \eqref{eq:cartesianProduct}, we have \eqref{eq:vi_split}.
\end{proof}

\subsubsection*{Part 2: component-by-component formulation}

\begin{lemm}\label{lema:proyeccion_discreta}
Let $\mathcal{C}\subset\mathbb{R}^{M}$ be a nonempty, closed, and convex set, and let $\langle\cdot,\cdot\rangle$ be any inner product on $\mathbb{R}^{M}$. For $a\in\mathcal{C}$ and $r\in\mathbb{R}^{M}$, the following statements are equivalent:
\begin {itemize}
\item[i)] $\langle r,\,\xi-a\rangle\ge0$ for all $\xi\in\mathcal{C}$;
\item[ii)] $a=\operatorname{Proj}_{\mathcal{C}}(a-\tau r)$ for all $\tau>0$.
\end{itemize}
\end{lemm}

\begin{proof}
Recall that, for $y\in\mathbb{R}^{M}$, the projection $p=\operatorname{Proj}_{\mathcal{C}}(y)$ is the unique element of $\mathcal{C}$ such that $\langle y-p,\ \xi-p\rangle\le0$ for all $\xi\in\mathcal{C}$. Applying this characterization with $y=a-\tau r$ and $p=a\in\mathcal{C}$, we obtain
$\langle a-\tau r-a,\ \xi-a\rangle=-\tau\langle r,\xi-a\rangle\le0$ for all $\xi\in\mathcal{C}$,
which, for $\tau>0$, is equivalent to (i). Since (i) does not involve $\tau$, the two statements are equivalent.
\end{proof}

\begin{remark}\label{obs:caja_discreta}
When $\mathcal{C}=[0,1]^{M}$ and $\langle\cdot,\cdot\rangle$ is the discrete $L^{2}$ inner product, $\langle x,y\rangle=h\sum_{j}x^{j}y^{j}$, condition (i) of Lemma \ref{lema:proyeccion_discreta} can be separated component by component. In fact, setting $\xi$ equal to $a$ except at the $j$-th component yields
\[
r^{j}(\xi-a^{j})\ge0\quad\forall\,\xi\in[0,1],\qquad j=1,\dots,M,
\]
and, therefore, $a^{j}=\operatorname{Proj}_{[0,1]}(a^{j}-\tau r^{j})$ for all $\tau>0$, where $\operatorname{Proj}_{[0,1]}$ is the pointwise projection defined in Section \ref{sec:notacion}. Equivalently,
\begin{equation}\label{eq:signos_discretos}
a^j = \begin{cases}
    0, &\text{if }r^j > 0, \\
    1, &\text{if }r^j < 0, \\
    \text{any value in }[0,1], &\text{if }r^j = 0.
\end{cases}
\end{equation}
\end{remark}

\subsubsection*{Part 3: Existence of the multiplier and conditions for optimality}

\begin{theo}[Discrete first-order necessary optimality conditions]\label{teo:kkt_discreto}
Let $0 < c_{\max} < T$, and let $[\overline{\ControlC}, \overline{\ControlS}] \in \mathcal{U}_{ad}^{h}$ be a local optimal control of $(P_h)$. Let $[\overline u, \overline\sigma]$ be the associated discrete states, given by \eqref{eq:esquema_estados}, let $[\overline w,\overline z]$ be the associated discrete adjoint states, given by \eqref{eq:adjunto_terminal}--\eqref{eq:adjunto}, and let $g_\ControlC, g_\ControlS \in \mathbb{R}^{N}$ be defined by \eqref{eq:gradiente_gc_gs}. Then there exists a scalar $\mu\in\mathbb{R}$ with $\mu\ge0$ such that:
\begin{itemize}
\item[a)] $\displaystyle \mu\Bigl(h\sum_{j=1}^{N}\overline{\ControlC}^{\,j}-c_{\max}\Bigr)=0$;
\item[b)] for all $j=1,\dots,N$,
      \[
      \bigl(g_\ControlC^{\,j}+\mu\bigr)\bigl(\xi-\overline{\ControlC}^{\,j}\bigr)\ge0\quad\forall\,\xi\in[0,1],
      \]
      equivalently
      \[
      \overline{\ControlC}^{\,j}=\operatorname{Proj}_{[0,1]}\Bigl(\overline{\ControlC}^{\,j}-\tau\bigl(g_\ControlC^{\,j}+\mu\bigr)\Bigr)\ \ \forall\,\tau>0;
      \]
\item[c)] for all $j=1,\dots,N$,
      \[
      g_\ControlS^{\,j}\bigl(\xi-\overline{\ControlS}^{\,j}\bigr)\ge0\quad\forall\,\xi\in[0,1],
      \]
      equivalently
      \[
      \overline{\ControlS}^{\,j}=\operatorname{Proj}_{[0,1]}\bigl(\overline{\ControlS}^{\,j}-\tau g_\ControlS^{\,j}\bigr)\ \ \forall\,\tau>0.
      \]
\end {itemize}
Explicitly, in terms of the state and adjoint variables, for all $j=1,\dots,N$,
\begin{equation}\label{eq:bangbang_discreto}
\overline{\ControlC}^{\,j}=
\begin{cases}
0, & \text{if } \kappa\,\overline u^{\,j}\overline\sigma^{\,j-1}\overline w^{\,j}<k_3^{j}+\mu,\\
1, & \text{if } \kappa\,\overline u^{\,j}\overline\sigma^{\,j-1}\overline w^{\,j}>k_3^{j}+\mu,\\
\text{any value in }[0,1], & \text{if } \kappa\,\overline u^{\,j}\overline\sigma^{\,j-1}\overline w^{\,j}=k_3^{j}+\mu,
\end{cases}
\end{equation}
\begin{equation}\label{eq:bangbang_discreto_s}
\overline{\ControlS}^{\,j}=
\begin{cases}
0, & \text{if } S_{c}\,\overline u^{\,j}\overline z^{\,j}<k_4^{j},\\
1, & \text{if } S_{c}\,\overline u^{\,j}\overline z^{\,j}>k_4^{j},\\
\text{any value in }[0,1], & \text{if } S_{c}\,\overline u^{\,j}\overline z^{\,j}=k_4^{j}.
\end{cases}
\end{equation}
\end{theo}
\begin{proof}
\textit{Step 1 (the multiplier).} By the Lemma \ref{lema:vi_discreta}, we have the two decoupled inequalities \eqref{eq:vi_split}. We apply the Kuhn–Tucker Theorem stated below (Theorem \ref{teo:kkt}) with
\[
X:=\mathbb{R}^{N},\qquad \Omega:=\mathcal{K}^{h}=[0,1]^{N},\qquad
f(v):=h\sum_{j=1}^{N}g_\ControlC^{\,j}v^{j},\qquad
G(v):=h\sum_{j=1}^{N}v^{j}-c_{\max}.
\]
Indeed, $\mathbb{R}^{N}$ is a vector space; $\mathcal{K}^{h}$ is convex; $f$ is linear and $G$ is affine on $\mathcal{K}^{h}$, so both are convex functionals with real values on $\Omega$; the point $x_1 := 0 \in \mathcal{K}^{h}$ satisfies $G(x_1) = -c_{\max} < 0$, because $c_{\max} > 0$; and $\mu_0=\inf\{f(v):v\in\mathcal{K}^{h},\ G(v)\le0\}=\inf_{v\in\mathcal{K}_\ControlC^{h}}f(v)$ is finite, since $f$ is continuous and $\mathcal{K}_\ControlC^{h}$ is nonempty and compact. Furthermore, the first inequality in \eqref{eq:vi_split} (the variational inequality) states precisely that
\[
f(\overline{\ControlC})\le f(v)\qquad\text{for all }v\in\mathcal{K}_\ControlC^{h}
=\{v\in\mathcal{K}^{h}:\ G(v)\le0\},
\]
hence the infimum $\mu_0$ is attained at $x_0=\overline{\ControlC}\in\mathcal{K}^{h}$, which satisfies $G(x_0)\le0$. Therefore, there exists a scalar $\mu\ge0$ such that $\mu\,G(\overline{\ControlC})=0$ and such that $\overline{\ControlC}$ is a minimizer of $f+\mu G$ on $\mathcal{K}^{h}$, that is,
\begin{equation}\label{eq:lagrangiano_discreto}
h\sum_{j=1}^{N}\bigl(g_\ControlC^{\,j}+\mu\bigr)\bigl(v^{j}-\overline{\ControlC}^{\,j}\bigr)\ \ge\ 0
\qquad\text{for all }v\in\mathcal{K}^{h}=[0,1]^{N}.
\end{equation}

\textit{Step 2 (componentwise formulation).} The inequality \eqref{eq:lagrangiano_discreto} is condition (i) of the Lemma \ref{lema:proyeccion_discreta} with $\mathcal{C}=[0,1]^{N}$, $a=\overline{\ControlC}$ and $r=g_\ControlC+\mu\mathbf{1}$. By Remark \ref{obs:caja_discreta}, it is decoupled component by component and is equivalent to b) and to \eqref{eq:bangbang_discreto}, since $g_\ControlC^{\,j}+\mu>0$ is equivalent to $k_3^{j}+\mu>\kappa\overline u^{\,j}\overline\sigma^{\,j-1}\overline w^{\,j}$. Applying the same argument to the second inequality in \eqref{eq:vi_split} yields c) and \eqref{eq:bangbang_discreto_s}.\\
\end{proof}

\begin{remark}[Structure of Discrete Optimal Controls]\label{obs:bangbang_discreto}
The Formulas \eqref{eq:bangbang_discreto}--\eqref{eq:bangbang_discreto_s} show that the optimal controls are of bang--bang type, except possibly at the nodes where the switching function vanishes. This structure is induced by the terms $k_3\ControlC$ and $k_4\ControlS$, which are linear in the controls.
\end{remark}

\begin{theo}[{\cite[Chapter 8, Theorem 1]{luenberger}}]\label{teo:kkt}
Let $X$ be a linear vector space and let $\Omega$ be a convex
subset of $X$. Let $f$ and $G$ be real-valued convex functionals on
$\Omega$. Assume that there exists a point $x_{1}\in\Omega$ such that
$G(x_{1})<0$. Let
\begin{equation}\label{eq:kkt-2}
\mu_{0}
=
\inf\bigl\{
    f(x):x\in\Omega,\;G(x)\leq 0
\bigr\},
\end{equation}
and assume that $\mu_{0}$ is finite. Then there exists a scalar
$\lambda_{0}\geq 0$ such that
\begin{equation}\label{3}
\mu_{0}
=
\inf_{x\in\Omega}
\left\{
    f(x)+\lambda_{0}G(x)
\right\}.
\end{equation}
Furthermore, if the infimum in \eqref{eq:kkt-2} is attained by some
$x_{0}\in\Omega$ satisfying $G(x_{0})\leq 0$, then the infimum in
\eqref{3} is also attained by $x_{0}$, and
\[
\lambda_{0}G(x_{0})=0.
\]
\end{theo}

\subsection{The projected gradient algorithm with Armijo line search}\label{sec:algoritmo_ejemplo}

We now particularize the classical method of Section \ref{sec:algoritmo_clasico} to the discrete problem. In this case, we use the projected gradient method with line search and select the step size by backtracking. This is the algorithm used in all the numerical experiments in this section, and it also produces the labels for the direct strategy.\\

Let $\eta,\nu\in(0,1)$, let $\omega_0 > 0$ be the initial step size, let $\texttt{tol}>0$ be a tolerance, let $\mathcal{N}\in\mathbb{N}$ be the number of consecutive iterations used to verify the stabilization of the functional, and let \texttt{maxIter} and \texttt{maxArmijo} be the maximum number of iterations of the algorithm and Armijo's attempts, respectively; their values are collected in Table \ref{tab:ParametrosArmijo}. Given the data \eqref{eq:identificacion_lambda} and an initial pair of controls $[\ControlC_{0},\ControlS_{0}]\in\mathcal{U}_{ad}^{h}$, the sequence $\{[\ControlC_{k},\ControlS_{k}]\}_{k\geq0}$ is constructed as follows.\\

\noindent
\textbf{Step 1. (State and adjoint state)} Given the current controls $[\ControlC_{k},\ControlS_{k}]\in\mathcal{U}_{ad}^{h}$, compute the discrete states $[u_{k},\sigma_{k}]$ by solving the scheme \eqref{eq:esquema_estados} forward in time, and then the discrete adjoint states $[w_{k},z_{k}]$ by solving \eqref{eq:adjunto_terminal}--\eqref{eq:adjunto} backward in time.\\

\noindent
\textbf{Step 2. (Cost functional)} Compute $J_k:=\widehat{J}_{h}(\ControlC_{k},\ControlS_{k})=J_h(u_{k},\sigma_{k},\ControlC_{k},\ControlS_{k})$ by means of \eqref{eq:funcional_discreto}.\\

\noindent
\textbf{Step 3. (Stopping criterion)} If $k\geq\mathcal{N}$, define $\Delta_{k-j}:=|J_{k-j}-J_{k-j-1}|$, $j=0,\dots,\mathcal{N}-1$, and stop whenever
\[
\max_{0\,\leq \, j\,\leq\,\mathcal{N}-1}\Delta_{k-j}<\texttt{tol},
\]
that is, whenever the value of the cost functional has remained stable during $\mathcal{N}$ consecutive iterations. Stop also if $k\geq\texttt{maxIter}$.\\

\noindent
\textbf{Step 4. (Reduced gradient)} Compute the discrete $L^2$ gradient $g_{k}:=\nabla_h\widehat{J}_{h}(\ControlC_{k},\ControlS_{k})$ by means of Proposition \ref{prop:gradiente_discreto}, that is, componentwise,
\[
(g_{k})_c^{\,j}=k_3^{j}-\kappa\,u_{k}^{j}\sigma_{k}^{j-1}w_{k}^{j},\qquad
(g_{k})_s^{\,j}=k_4^{j}-S_{c}\,u_{k}^{j}z_{k}^{j},\qquad j=1,\dots,N,
\]
so that no sensitivity system has to be solved.\\

\noindent
\textbf{Step 5. (Projected descent with Armijo line search)} For $j=0,1,\dots,\texttt{maxArmijo}$, define the projected trial point by
\[
q_{k,j}:=\operatorname{Proj}_{\mathcal{U}_{ad}^{h}}\bigl([\ControlC_{k},\ControlS_{k}]-\omega_0 \nu^{j}g_{k}\bigr).
\]
Then set
\[
j_k:=\min\Bigl\{j=0,1,\dots,\texttt{maxArmijo}:\ 
\widehat{J}_{h}(q_{k,j})-J_k\leq-\eta\omega_0 \nu^j \|g_k\|^2_h\,
\Bigr\}.
\]
Set $\omega_k:=\omega_0 \nu^{\,j_k}$ (see \cite{Polak}) and set
\[
[\ControlC_{k+1},\ControlS_{k+1}]
=\operatorname{Proj}_{\mathcal{U}_{ad}^{h}}\bigl([\ControlC_k,\ControlS_k]-\omega_k g_k\bigr)
=q_{k,j_k}.
\]
If no $j\in\{0,1,\dots,\texttt{maxArmijo}\}$ satisfies the above inequality, we keep $[\ControlC_{k+1},\ControlS_{k+1}]=[\ControlC_{k},\ControlS_{k}]$; in that case the value of the cost functional does not change any more and the test of Step 3 is activated after $\mathcal{N}$ iterations.\\

\noindent
\textbf{Step 6.} Set $k\leftarrow k+1$ and return to Step 1.\\

Finally, the control returned by the algorithm is submitted to the validation test \eqref{eq:filtro_general}, which is used here with $M=3$ test directions, namely
\[
v_1=[\mathbf{1},\mathbf{1}],\qquad v_2=-[\mathbf{1},\mathbf{1}],\qquad v_3=-\frac{g^{*}}{\Vert g^{*}\Vert_{\infty}},
\]
where $\mathbf{1}=(1,\dots,1)\in\mathbb{R}^{N}$ and $g^{*}$ denotes the reduced discrete $L^2$ gradient of Step 4 evaluated on the returned controls $[\ControlC^{*},\ControlS^{*}]$; the third direction is omitted when $g^{*}=0$ (in this case the algorithm must stop). Each perturbed control is projected back onto $\mathcal{U}_{ad}^{h}$, so that all three trials are admissible. The two constant directions verify that the returned controls resist a global increase and a global decrease in both therapies, while the third probes the first-order optimality condition and is the one that detects the premature stops of Step 3, in which the Armijo line search has exhausted its attempts without producing a sufficient decrease in the functional. In the experiments shown below, we take $\delta=10^{-2}$ and $\texttt{tol}_{\mathrm{op}}=10^{-5}$.

\subsection{The learned operators}\label{sec:operadores}

The \textit{indirect strategy} relies on two independent DeepONets. The first one approximates the state operator
\begin{equation}\label{eq:Ge}
\mathcal{G}_{e}:[\ControlC,\ControlS,u_0,\sigma_0]\longmapsto[u,\sigma],
\end{equation}
whereas the second one is designed to approximate the operator associated with the adjoint variables,
\begin{equation}\label{eq:Ga}
\mathcal{G}_{a}:[\ControlC,\ControlS,u_0,\sigma_0,l_1,l_2,k_1,k_2]\longmapsto[w,z].
\end{equation}
The parameters of the model, as well as the desired oxygen $\sigma_d$ and $\sigma_{dT}$ remain fixed. In both cases, the controls $\ControlC$ and $\ControlS$ are represented by their values at $m$ \textit{equispaced} sensors in $[0,T]$. Unlike the usual DeepONet, whose output is a single real value, each one of these networks approximates an operator with two components; that is, following Remark \ref{rem:multisalida}, the branch net produces $2\ell$ coefficients, rearranged as two vectors of latent dimension $\ell$, one per component, whereas the trunk net produces a single vector of $\ell$ components for the queried time, shared by both components. Each component is obtained by its own inner product between the corresponding vector of the branch net and the vector of the trunk net, plus a trainable bias; in this way, a single evaluation of the state network returns $[u(t),\sigma(t)]$ at the queried time $t\in[0,T]$, and a single evaluation of the adjoint network returns $[w(t),z(t)]$ at that same time $t\in[0,T]$. Thus, the input of the state network has dimension $2m+2$ (two controls together with two scalar input values), whereas the input of the adjoint network has dimension $2m+6$ (two controls together with six scalar input values).\\

Figure \ref{fig:esquema_deeponet_edo} shows the role of the networks inside the optimal control algorithm of Section \ref{sec:algoritmo_ejemplo}, in which the forward and backward resolutions of Step 1 are replaced by evaluations of the trained DeepONets.

\begin{figure}[H]
\centering
\resizebox{0.8\linewidth}{!}{%
\begin{tikzpicture}[x=1cm,y=1cm,line cap=round,line join=round]
  \definecolor{blueacc}{RGB}{20,70,255}
  \tikzset{
    flow/.style={blueacc,-{Latex[length=3mm,width=2mm]},line width=0.9pt},
    loopflow/.style={blueacc,dashed,-{Latex[length=3mm,width=2mm]},line width=0.9pt,
      rounded corners=8pt},
    netline/.style={draw=black,line width=0.4pt},
    neuronB/.style={circle,draw=black,line width=0.55pt,fill=white,
      minimum size=0.21cm,inner sep=0pt},
    neuronT/.style={circle,draw=black,line width=0.55pt,fill=white,
      minimum size=0.21cm,inner sep=0pt},
    panel/.style={rounded corners=4pt,draw=black!60,line width=0.55pt},
    inlab/.style={rounded corners=2pt,draw=black!60,line width=0.55pt,fill=white,
      align=center,font=\small,inner sep=4pt},
    stepbox/.style={rounded corners=2pt,draw=black!60,line width=0.55pt,fill=white,
      align=center,font=\small,inner sep=5pt,minimum height=0.95cm},
    result/.style={rounded corners=2pt,draw=blueacc,line width=0.8pt,fill=white,
      align=center,font=\small,inner sep=5pt,minimum height=0.95cm},
    odot/.style={circle,draw=blueacc,line width=0.9pt,fill=white,
      minimum size=0.55cm,inner sep=0pt,font=\small},
    nettag/.style={font=\small\itshape},
    notetag/.style={font=\tiny,text=black!60},
    ptitle/.style={font=\small\bfseries}
  }

  % Miniature neural network: layers 3-4-3
  \newcommand{\rednet}[4]{% #1 prefix, #2 x, #3 y, #4 neuron style
    \foreach \i/\dy in {1/-0.36,2/0,3/0.36}{
      \coordinate (#1A\i) at (#2,#3+\dy);
      \coordinate (#1C\i) at (#2+1.5,#3+\dy);}
    \foreach \i/\dy in {1/-0.54,2/-0.18,3/0.18,4/0.54}{
      \coordinate (#1B\i) at (#2+0.75,#3+\dy);}
    \foreach \i in {1,2,3}{\foreach \j in {1,...,4}{
      \draw[netline] (#1A\i)--(#1B\j);
      \draw[netline] (#1C\i)--(#1B\j);}}
    \foreach \i in {1,2,3}{\node[#4] at (#1A\i){};\node[#4] at (#1C\i){};}
    \foreach \j in {1,...,4}{\node[#4] at (#1B\j){};}
  }

  % Panel: state DeepONet
  \draw[panel] (-0.2,1.05) rectangle (6.1,4.25);
  \node[ptitle,anchor=south] at (2.95,4.30) {State DeepONet\; $\mathcal{G}_{e}$};
  \rednet{sb}{1.0}{3.40}{neuronB}
  \rednet{st}{1.0}{1.90}{neuronT}
  \node[nettag,rotate=90] at (0.45,3.40) {branch};
  \node[nettag,rotate=90] at (0.45,1.90) {trunk};
  \node[odot] (sdotu) at (4.5,2.98) {$\odot$};
  \node[odot] (sdots) at (4.5,2.32) {$\odot$};
  \draw[flow,shorten <=3pt] (sbC2) -- (sdotu);
  \draw[flow,shorten <=3pt] (sbC2) -- (sdots);
  \draw[flow,shorten <=3pt] (stC2) -- (sdotu);
  \draw[flow,shorten <=3pt] (stC2) -- (sdots);
  \node[notetag] at (4.5,1.80) {two inner products};

  % Panel: adjoint DeepONet
  \draw[panel] (-0.2,-2.85) rectangle (6.1,0.35);
  \node[ptitle,anchor=south] at (2.95,0.40) {Adjoint DeepONet\; $\mathcal{G}_{a}$};
  \rednet{ab}{1.0}{-0.50}{neuronB}
  \rednet{at}{1.0}{-2.00}{neuronT}
  \node[nettag,rotate=90] at (0.45,-0.50) {branch};
  \node[nettag,rotate=90] at (0.45,-2.00) {trunk};
  \node[odot] (adotu) at (4.5,-0.92) {$\odot$};
  \node[odot] (adots) at (4.5,-1.58) {$\odot$};
  \draw[flow,shorten <=3pt] (abC2) -- (adotu);
  \draw[flow,shorten <=3pt] (abC2) -- (adots);
  \draw[flow,shorten <=3pt] (atC2) -- (adotu);
  \draw[flow,shorten <=3pt] (atC2) -- (adots);
  \node[notetag] at (4.5,-2.10) {two inner products};

  % Inputs
  \node[inlab,anchor=east] (in1) at (-1.0,3.40) {$\ControlC,\,\ControlS,\,u_0,\,\sigma_0$};
  \node[inlab,anchor=east] (in2) at (-1.0,1.90) {time $t$};
  \node[inlab,anchor=east] (in3) at (-1.0,-0.50)
    {$\ControlC,\,\ControlS,\,u_0,\,\sigma_0$\\[-1pt]$l_1,\,l_2,\,k_1,\,k_2$};
  \node[inlab,anchor=east] (in4) at (-1.0,-2.00) {time $t$};
  \draw[flow] (in1.east) -- (-0.2,3.40);
  \draw[flow] (in2.east) -- (-0.2,1.90);
  \draw[flow] (in3.east) -- (-0.2,-0.50);
  \draw[flow] (in4.east) -- (-0.2,-2.00);

  % Outputs
  \node[result] (statesout) at (8.2,2.65) {states\\$u(t),\ \sigma(t)$};
  \node[result] (adjout) at (8.2,-1.25) {adjoint states\\$w(t),\ z(t)$};
  \draw[flow] (sdotu.east) -- (statesout.west |- sdotu);
  \draw[flow] (sdots.east) -- (statesout.west |- sdots);
  \draw[flow] (adotu.east) -- (adjout.west |- adotu);
  \draw[flow] (adots.east) -- (adjout.west |- adots);
  \node[font=\scriptsize] at (5.5,3.20) {$u(t)$};
  \node[font=\scriptsize] at (5.5,2.54) {$\sigma(t)$};
  \node[font=\scriptsize] at (5.5,-0.70) {$w(t)$};
  \node[font=\scriptsize] at (5.5,-1.36) {$z(t)$};

  % Stages of the gradient algorithm
  \node[stepbox,anchor=east] (evalJ) at (9.7,-4.1) {evaluation of $J_h$\\and of the gradient};
  \node[stepbox] (armijo) at (4.05,-4.1) {Armijo line search\\and projection};
  \node[stepbox,draw=blueacc,line width=0.8pt,anchor=west] (newc) at (-2.6,-4.1)
    {updated controls\\$\ControlC^{\,k+1},\ \ControlS^{\,k+1}$};
  \draw[flow,rounded corners=8pt] (statesout.east) -- (10.2,2.65)
    -- (10.2,-4.1) -- (evalJ.east);
  \draw[blueacc,line width=0.9pt] (adjout.east) -- (10.2,-1.25);
  \fill[blueacc] (10.2,-1.25) circle (1.7pt);
  \draw[flow] (evalJ.west) -- (armijo.east);
  \draw[flow] (armijo.west) -- (newc.east);

  % Feedback
  \draw[loopflow] (newc.west) -- (-3.9,-4.1) -- (-3.9,3.40) -- (in1.west);
  \draw[loopflow] (-3.9,-0.50) -- (in3.west);
  \fill[blueacc] (-3.9,-0.50) circle (1.7pt);
  
  \node[notetag,font=\scriptsize,text=black!60] at (3.55,-5.05)
    {the steps of the gradient algorithm are unchanged. The loop is repeated
     until the stopping criterion. };
\end{tikzpicture}
}
\caption{Integration of the supervised DeepONets into the optimal control algorithm (\textit{indirect strategy}).}
\label{fig:esquema_deeponet_edo}
\end{figure}
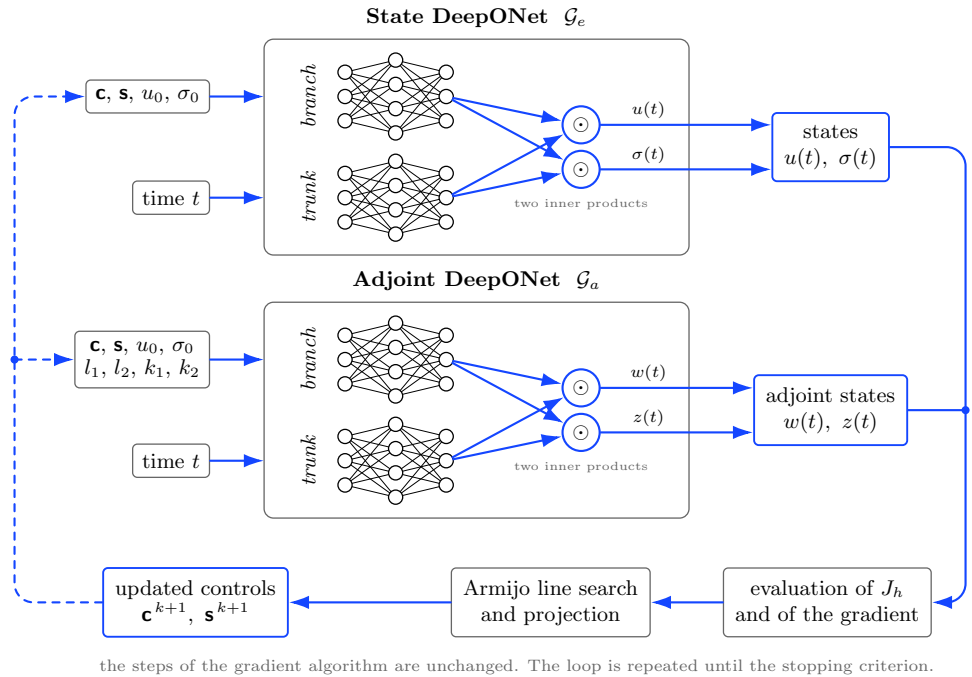

The \textit{direct strategy}, in contrast, employs a single DeepONet which approximates the operator
\begin{equation}\label{eq:Gc}
\mathcal{G}_{c}:[\ControlC,\ControlS,u_0,\sigma_0,l_1,l_2,k_1,k_2,k_3,k_4,c_{\max}]\longmapsto[\ControlC^{*},\ControlS^{*}],
\end{equation}
where $[\ControlC,\ControlS]$ are the \textit{initial} controls with which the descent algorithm of Section \ref{sec:algoritmo_ejemplo} starts and $[\ControlC^{*},\ControlS^{*}]$ is the pair of optimal controls that this algorithm produces. The input is larger than that of the previous operators because the resolution of the optimal control problem depends also on the weights $k_3,k_4$ and on the bound $c_{\max}$. Consequently, the input of the branch net has dimension $2m+9$ (two controls together with nine scalar input values). The output is again organized in two components, one per control, with the same construction of two inner products described above, so that a single evaluation of the network returns $[\ControlC^{*}(t),\ControlS^{*}(t)]$ at the queried time $t\in[0,T]$.\\

Figure \ref{fig:esquema_deeponet_directa_edo} illustrates this second strategy. Unlike Figure \ref{fig:esquema_deeponet_edo}, the stages of the optimal control algorithm no longer appear, since the network approximates the optimal controls directly.

\begin{figure}[H]
\centering
\resizebox{0.85\linewidth}{!}{%
\begin{tikzpicture}[x=1cm,y=1cm,line cap=round,line join=round]
  \definecolor{blueacc}{RGB}{20,70,255}
  \tikzset{
    flow/.style={blueacc,-{Latex[length=3mm,width=2mm]},line width=0.9pt},
    netline/.style={draw=black,line width=0.4pt},
    neuronB/.style={circle,draw=black,line width=0.55pt,fill=white,
      minimum size=0.21cm,inner sep=0pt},
    panel/.style={rounded corners=4pt,draw=black!60,line width=0.55pt},
    inlab/.style={rounded corners=2pt,draw=black!60,line width=0.55pt,fill=white,
      align=center,font=\small,inner sep=4pt},
    result/.style={rounded corners=2pt,draw=blueacc,line width=0.8pt,fill=white,
      align=center,font=\small,inner sep=5pt,minimum height=0.95cm},
    odot/.style={circle,draw=blueacc,line width=0.9pt,fill=white,
      minimum size=0.55cm,inner sep=0pt,font=\small},
    nettag/.style={font=\small\itshape},
    notetag/.style={font=\tiny,text=black!60},
    ptitle/.style={font=\small\bfseries}
  }

  % Miniature neural network: layers 3-4-3
  \newcommand{\redneta}[4]{% #1 prefix, #2 x, #3 y, #4 neuron style
    \foreach \i/\dy in {1/-0.36,2/0,3/0.36}{
      \coordinate (#1A\i) at (#2,#3+\dy);
      \coordinate (#1C\i) at (#2+1.5,#3+\dy);}
    \foreach \i/\dy in {1/-0.54,2/-0.18,3/0.18,4/0.54}{
      \coordinate (#1B\i) at (#2+0.75,#3+\dy);}
    \foreach \i in {1,2,3}{\foreach \j in {1,...,4}{
      \draw[netline] (#1A\i)--(#1B\j);
      \draw[netline] (#1C\i)--(#1B\j);}}
    \foreach \i in {1,2,3}{\node[#4] at (#1A\i){};\node[#4] at (#1C\i){};}
    \foreach \j in {1,...,4}{\node[#4] at (#1B\j){};}
  }

  % Panel: optimal-control DeepONet
  \draw[panel] (-0.2,1.05) rectangle (6.1,4.25);
  \node[ptitle,anchor=south] at (2.95,4.30) {Optimal-control DeepONet\; $\mathcal{G}_{c}$};
  \redneta{cb}{1.0}{3.40}{neuronB}
  \redneta{ct}{1.0}{1.90}{neuronB}
  \node[nettag,rotate=90] at (0.45,3.40) {branch};
  \node[nettag,rotate=90] at (0.45,1.90) {trunk};
  \node[odot] (cdotc) at (4.5,2.98) {$\odot$};
  \node[odot] (cdots) at (4.5,2.32) {$\odot$};
  \draw[flow,shorten <=3pt] (cbC2) -- (cdotc);
  \draw[flow,shorten <=3pt] (cbC2) -- (cdots);
  \draw[flow,shorten <=3pt] (ctC2) -- (cdotc);
  \draw[flow,shorten <=3pt] (ctC2) -- (cdots);
  \node[notetag] at (4.5,1.80) {two inner products};

  % Inputs
  \node[inlab,anchor=east] (din1) at (-1.0,3.40)
    {$\ControlC,\,\ControlS,\,u_0,\,\sigma_0,\,c_{\max}$\\[-1pt]$l_1,\,l_2,\,k_1,\,k_2,\,k_3,\,k_4$};
  \node[inlab,anchor=east] (din2) at (-1.0,1.90) {time $t$};
  \draw[flow] (din1.east) -- (-0.2,3.40);
  \draw[flow] (din2.east) -- (-0.2,1.90);

  % Outputs
  \node[result] (ctrlout) at (8.4,2.65) {optimal controls\\$\ControlC^*(t),\ \ControlS^*(t)$};
  \draw[flow] (cdotc.east) -- (ctrlout.west |- cdotc);
  \draw[flow] (cdots.east) -- (ctrlout.west |- cdots);
  \node[font=\scriptsize] at (5.6,3.20) {$\ControlC^*(t)$};
  \node[font=\scriptsize] at (5.6,2.54) {$\ControlS^*(t)$};

  \node[notetag,font=\scriptsize,text=black!60] at (3.4,0.55)
    {A single evaluation of the network. There is no iterative loop.};
\end{tikzpicture}
}
\caption{Direct approximation of the optimal controls by means of a single supervised DeepONet (\textit{direct strategy}).}
\label{fig:esquema_deeponet_directa_edo}
\end{figure}

\subsection{Computational setup and accuracy of the learned operators}\label{sec:configuracion}

In all the numerical experiments the final time is fixed to $T=2$, and the time interval is discretized with $N=400$ subintervals, so that $h=T/N=0.005$. The controls are represented by their values at $m=N=400$ \textit{equispaced} sensors, which coincide with the points $t_1,\dots,t_N$ of the time grid, so that the vector of sampled values \eqref{eq:vector_entrada_deeponet} of each control is exactly $(\ControlC^{1},\dots,\ControlC^{N})$, respectively $(\ControlS^{1},\dots,\ControlS^{N})$; the outputs of the three operators are stored at $400$ \textit{equispaced} observation points. The parameters of the state and adjoint systems are kept fixed at the values of Table \ref{tab:Parametros_edo}, the desired oxygen levels are $\sigma_d = \sigma_{dT} = 1$, and the projected gradient method with Armijo line search uses the parameters of Table \ref{tab:ParametrosArmijo}. The three DeepONets share the same architecture, except for the input dimension of the branch net, and were trained with the same loss function and optimizer. The complete architecture, the training hyperparameters, the sampling distributions used to build the supervised datasets and the cost of their generation are collected in Appendix \ref{ap:computacional}.

\begin{table}[H]
\centering
\begin{minipage}{0.43\textwidth}
\centering
\begin{tabular}{cc|cc}
\hline
\multicolumn{2}{c|}{\textbf{Tumor equation}} & \multicolumn{2}{c}{\textbf{Oxygen equation}} \\ \hline
\textbf{Parameter}      & \textbf{Value}     & \textbf{Parameter}      & \textbf{Value}     \\ \hline
$\alpha$                & $4$                & $A_{ox}$                & $0.15$              \\
$\widehat{\rho}$        & $3.5$              & $k_{ox}$                & $0.1$              \\
$b$                     & $1$                & $\beta$                 & $1$                \\
$\kappa$                & $1$                & $\gamma=P_{er}S_v$      & $0.1$                \\
                        &                    & $S_c$                   & $0.4$              \\ \hline
\end{tabular}
\caption{Parameters of the state and adjoint systems.}
\label{tab:Parametros_edo}
\end{minipage}
\hspace{0.08\textwidth}
\begin{minipage}{0.43\textwidth}
\centering
{\small
\begin{tabular}{c|c}
\hline
\textbf{Parameter} & \textbf{Value}\\ \hline
Maximum number of iterations & $800$ \\
Initial step size & $5\times 10^{-3}$ \\
Minimum step size & $10^{-8}$ \\
Step reduction factor $\nu$ & $0.5$ \\
Maximum number of reductions & $10$ \\
Armijo constant $\eta$ & $10^{-3}$ \\
Tolerance of the stopping criterion & $10^{-6}$ \\ \hline
\end{tabular}
}
\caption{Parameters of the projected gradient method and of the Armijo line search used in the update of the controls.}
\label{tab:ParametrosArmijo}
\end{minipage}
\end{table}
During the training of the state DeepONet the early stopping criterion was activated at iteration $474\,000$, whereas the adjoint DeepONet reached the maximum of $800\,000$ iterations and, in the \textit{direct strategy}, the optimal-control (from algorithm) DeepONet reached the maximum of $500\,000$ iterations. The observed training times were $2\,559.7$ s ($42$ min $40$ s), $3\,752.2$ s ($1$ h $2$ min $5$ s) and $2\,547.8$ s ($42$ min $28$ s), respectively. The validation set was evaluated every $2\,000$ steps and, at each of those evaluations, the parameters were stored whenever the validation error improved on the best value observed so far. Table \ref{tab:errores_deeponet_edo} presents the mean squared errors of the restored checkpoints on the two sets; the complete training histories are reported in Appendix \ref{ap:historias}.
\begin{table}[H]
\centering
\begin{tabular}{l|c|c}
\hline
\textbf{Quantity} & \textbf{Training MSE} & \textbf{Validation MSE} \\ \hline
\multicolumn{3}{c}{\textit{Indirect strategy}} \\ \hline
States, average & {$9.273880\times10^{-5}$} & {$9.850481\times10^{-5}$}  \\
$u$ & {$1.566716\times10^{-4}$} & {$1.646972\times10^{-4}$}  \\
$\sigma$ & {$2.880604\times10^{-5}$} & {$3.231240\times10^{-5}$} \\ \hline
Adjoint states, average & {$6.706016\times10^{-2}$} & {$1.808690\times10^{-1}$} \\
$w$ & {$5.277055\times10^{-2}$} & {$1.965048\times10^{-1}$}\\
$z$ & {$8.134978\times10^{-2}$} & {$1.652333\times10^{-1}$} \\ \hline
\multicolumn{3}{c}{\textit{Direct strategy}} \\ \hline
Optimal controls, average & $1.465341\times10^{-3}$ & $2.727353\times10^{-3}$ \\
$\ControlC^{*}$ & $1.511438\times10^{-3}$ & $2.627751\times10^{-3}$ \\
$\ControlS^{*}$ & $1.419244\times10^{-3}$ & $2.826954\times10^{-3}$ \\ \hline
\end{tabular}
\caption{Mean squared errors of the restored checkpoints of the DeepONets on the training and validation datasets.}
\label{tab:errores_deeponet_edo}
\end{table}

\subsection{Comparison on the optimal control problem}\label{sec:comparacion}

The questions in which we are interested now are the following:\\

\textit{When one of the two strategies based on DeepONets is used, is the solution of the control problem preserved? With which accuracy? How do they differ in terms of performance?}\\

In order to answer them, the same test set of $100$ random samples was solved with the three methods: the classical method of Section \ref{sec:algoritmo_ejemplo}, the indirect strategy, which runs the algorithm of Section \ref{sec:algoritmo_clasico} replacing the two solvers by the state and adjoint DeepONets, and the direct strategy, which consists in evaluating once the optimal-control DeepONet. For each sample, the bound $c_{\max}$ was drawn uniformly on $[0.05,2)$ and the six weights $l_1,l_2,k_1,k_2,k_3,k_4$ on $\{0,\dots,20\}$, following the sampling of Table \ref{tab:muestreo_deeponet_edo}; moreover, the initial control $\ControlC$ was projected onto the corresponding admissible set. The samples were generated with a seed different from the ones used in training and validation, and the three methods were run on the same $100$ samples and on the same computer. Next, the optimal control obtained with the classical method is taken as the reference.

\subsubsection*{Structure of the controls over the $100$ samples}

For the $100$ samples, Figure \ref{fig:errores_prueba_deeponet_a} shows the mean squared error per sample of each strategy with respect to the classical reference. In both cases the distribution of the error is strongly asymmetric, with most of the samples concentrated at small values and a few isolated peaks, which indicates a good approximation of the optimal control on these samples. In particular, for the \textit{indirect strategy}, $56$ of $100$ samples have an MSE in $\ControlC$ smaller than $10^{-3}$ and $93$ smaller than $10^{-2}$; analogously, for $\ControlS$, $83$ samples present an MSE smaller than $10^{-3}$ and $93$ smaller than $10^{-2}$. For the \textit{direct strategy}, $40$ samples lie below $10^{-3}$ in $\ControlC$ and $91$ below $10^{-2}$, whereas for $\ControlS$ these numbers are $66$ and $94$, respectively. \\

According to Figure \ref{fig:errores_prueba_deeponet_a}, the \textit{indirect strategy} is more accurate for most samples, since its median MSE in $\ControlC$ is $7.46\times10^{-4}$ compared to $1.40\times10^{-3}$ for the \textit{direct strategy}, and in $\ControlS$ it is $7.41\times10^{-5}$ compared to $3.18\times10^{-4}$. This is consistent with the fact that it preserves the optimization algorithm and only approximates the state and adjoint variables used by the algorithm, so that the projection onto the feasible set and the minimization of the functional continue to take effect at each iteration.
\newcommand{\imgdoscolerr}{0.47\textwidth}
\pgfplotsset{
estilofig4/.style={
scale only axis,
width=5.85cm, height=5.3cm,
xmin=0, xmax=101,
ymin=2e-8, ymax=6e-1,
xlabel={Sample number}, ylabel={MSE},
grid=both, ymajorgrids, grid style={gray!30},
legend cell align=left,
legend style={font=\footnotesize, at={(0.03,0.03)}, anchor=south west,
fill=white, fill opacity=0.92, text opacity=1, draw=gray!45,
rounded corners=2pt, inner xsep=5pt, inner ysep=3pt},
/pgf/number format/use comma=false,
/pgf/number format/set decimal separator={.},
filter discard warning=false, unbounded coords=discard,
tick label style={font=\small}, label style={font=\small},
title style={font=\small}
}}

\begin{figure}[H]
\centering
\begin{subfigure}[t]{\imgdoscolerr}
\centering
\begin{tikzpicture}
\begin{axis}[estilofig4, ymode=log, title={MSE of $\ControlC^{*}$}]
\addplot+[color={rgb,255:red,21;green,101;blue,192}, solid, line width=1.1pt, mark=none]
table[x=sample,y=mseC1] {figuras/errores_prueba.dat};
\addlegendentry{Indirect}
\addplot+[color={rgb,255:red,230;green,81;blue,0}, dashed, line width=1.1pt, mark=none]
table[x=sample,y=mseC2] {figuras/errores_prueba.dat};
\addlegendentry{Direct}
\end{axis}
\end{tikzpicture}
\end{subfigure}\hfill
\begin{subfigure}[t]{\imgdoscolerr}
\centering
\begin{tikzpicture}
\begin{axis}[estilofig4, ymode=log, title={MSE of $\ControlS^{*}$}]
\addplot+[color={rgb,255:red,21;green,101;blue,192}, solid, line width=1.1pt, mark=none]
table[x=sample,y=mseS1] {figuras/errores_prueba.dat};
\addlegendentry{Indirect}
\addplot+[color={rgb,255:red,230;green,81;blue,0}, dashed, line width=1.1pt, mark=none]
table[x=sample,y=mseS2] {figuras/errores_prueba.dat};
\addlegendentry{Direct}
\end{axis}
\end{tikzpicture}
\end{subfigure}
\caption{Mean squared error per sample (logarithmic scale), for the optimal control $\ControlC$ (left) and $\ControlS$ (right). This represents the discrepancy between the optimal controls obtained with the classical solver and those obtained with each one of the two strategies, over the $100$ samples of the test set.}
\label{fig:errores_prueba_deeponet_a}
\end{figure}
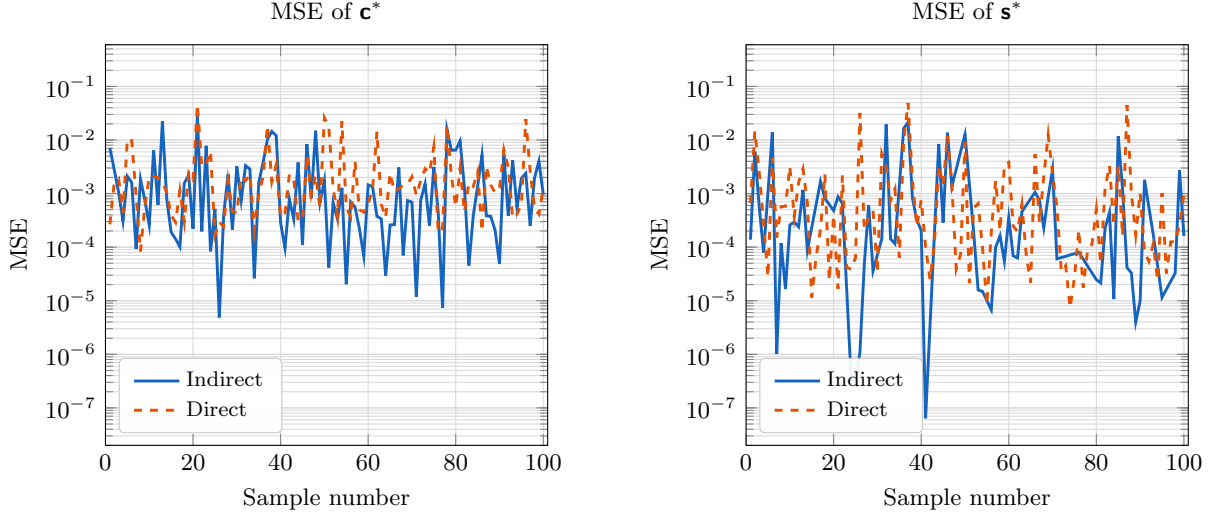

It is also interesting to observe the approximation of the optimal control on some specific samples. To this end, six samples were chosen at random from the test set mentioned above. The graphs of the control results obtained for each strategy in each random sample are shown in Appendix \ref{ap:samples}.

\subsubsection*{Admissibility of the controls}

The \textit{indirect strategy} structure retains the optimization algorithm and the projection algorithm, so its approximations are admissible; but it is important to point out that the optimal controls obtained from the \textit{direct strategy} are not always admissible. Over the $100$ samples of the test set, $99$ do not satisfy at all times the pointwise bounds $0\leq\ControlC^{*}\leq1$ and $100$ do not satisfy them for $\ControlS^{*}$, with maximum excesses of $0.249$ and $0.227$ and average excesses of $0.050$ and $0.029$, respectively; moreover, $54$ samples do not satisfy the integral constraint, although in this last case the excess is small, with a maximum value of $0.059$ and a median of $0.007$ over an interval of length $T=2$. Figure \ref{fig:admisibilidad_prueba_deeponet} quantifies this behavior, showing the percentage of samples whose controls satisfy each constraint up to a tolerance $\varepsilon$. The two curves of the \textit{direct strategy} start at $1\%$ and $0\%$ for the pointwise bounds and at $46\%$ for the integral constraint, and reach $100\%$ only for tolerances of order $0.25$ in $\ControlC^*$ and $0.061$ for the integral constraint. This is a direct consequence of the fact that the constraints defining $\mathcal{U}_{ad}^{h}$ are never imposed on the network output; in addition, this demonstrate the effectiveness and capacity of the \textit{direct strategy} to approximate these global and specific constraints.
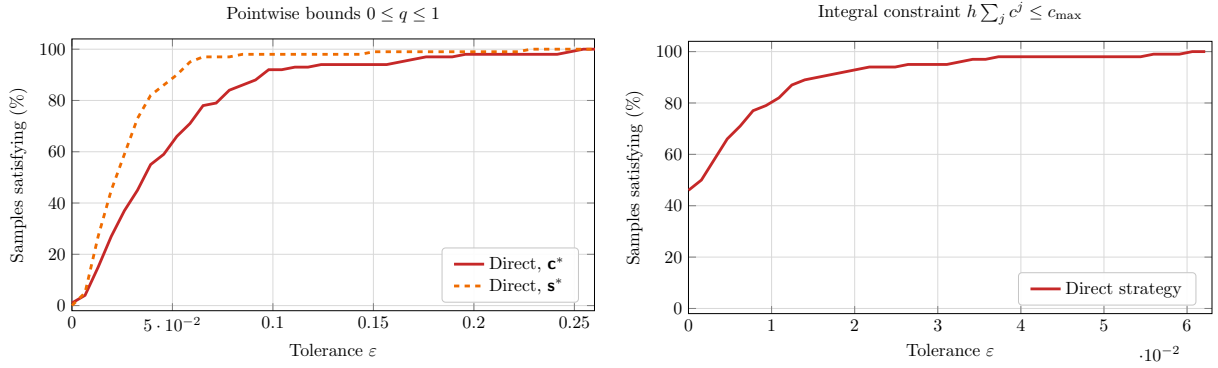
\begin{figure}[H]
\centering
\begin{subfigure}[t]{0.49\textwidth}
\centering
\resizebox{\linewidth}{!}{%
\begin{tikzpicture}
\begin{axis}[
width=12cm, height=7cm,
title={Pointwise bounds $0\leq q\leq1$},
xmin=0, xmax=0.26, ymin=-2, ymax=104,
xtick={0,0.05,0.10,0.15,0.20,0.25},
xlabel={Tolerance $\varepsilon$}, ylabel={Samples satisfying (\%)},
grid=both, ymajorgrids, grid style={gray!30},
legend cell align=left,
legend style={font=\normalsize, at={(0.97,0.03)}, anchor=south east,
fill=white, fill opacity=0.92, text opacity=1, draw=gray!45,
rounded corners=2pt, inner xsep=6pt, inner ysep=3pt},
/pgf/number format/use comma=false,
/pgf/number format/set decimal separator={.},
tick label style={font=\normalsize}, label style={font=\normalsize},
title style={font=\normalsize}
]
\addplot+[color={rgb,255:red,198;green,40;blue,40}, solid, line width=1.6pt, mark=none]
table[x=eps,y=pctC2] {figuras/admisibilidad_puntual.dat};
\addlegendentry{Direct, $\ControlC^{*}$}
\addplot+[color={rgb,255:red,239;green,108;blue,0}, dashed, line width=1.6pt, mark=none]
table[x=eps,y=pctS2] {figuras/admisibilidad_puntual.dat};
\addlegendentry{Direct, $\ControlS^{*}$}
\end{axis}
\end{tikzpicture}%
}
\end{subfigure}
\hfill
\begin{subfigure}[t]{0.49\textwidth}
\centering
\resizebox{\linewidth}{!}{%
\begin{tikzpicture}
\begin{axis}[
width=12cm, height=7cm,
title={Integral constraint $h\sum_j c^{j}\leq c_{\max}$},
xmin=0, xmax=0.063, ymin=-2, ymax=104,
xlabel={Tolerance $\varepsilon$}, ylabel={Samples satisfying (\%)},
grid=both, ymajorgrids, grid style={gray!30},
legend cell align=left,
legend style={font=\normalsize, at={(0.97,0.03)}, anchor=south east,
fill=white, fill opacity=0.92, text opacity=1, draw=gray!45,
rounded corners=2pt, inner xsep=6pt, inner ysep=3pt},
/pgf/number format/use comma=false,
/pgf/number format/set decimal separator={.},
tick label style={font=\normalsize}, label style={font=\normalsize},
title style={font=\normalsize}
]
\addplot+[color={rgb,255:red,198;green,40;blue,40}, solid, line width=1.6pt, mark=none]
table[x=eps,y=pct2] {figuras/admisibilidad_integral.dat};
\addlegendentry{Direct strategy}
\end{axis}
\end{tikzpicture}%
}
\end{subfigure}
\caption{Percentage of the $100$ samples of the test set whose optimal controls satisfy each constraint up to a tolerance $\varepsilon$. The pointwise bounds (left graph) and integral constraint (right graph). The value at $\varepsilon=0$ is the percentage of samples that satisfy the constraint exactly.}
\label{fig:admisibilidad_prueba_deeponet}
\end{figure}

\subsubsection*{Computational cost}

Now we are now interested in answering the last question, concerning the computational cost. In Figure \ref{fig:tiempos_prueba_deeponet} the execution times are compared. In particular, the left panel shows the execution times for each sample, where a very different dispersion between the three methods is observed, since the times of the classical method vary between $0.028$ s and $1.949$ s, those of the indirect strategy between $0.044$ s and $6.528$ s, and those of the direct strategy between $6.85\times10^{-4}$ s and $8.98\times10^{-4}$ s. The classical method and the indirect strategy inherit their variability from the number of iterations that the descent algorithm needs, which changes from one sample to another, whereas the direct strategy is practically constant, as it always performs a single evaluation of the network.\\

Additionally, the right panel shows the cumulative time for the three methods. When the $100$ samples have been completed, the classical implementation has taken $64.51$ s against $70.86$ s of the indirect strategy and $0.074$ s of the direct strategy. The direct strategy is therefore about $878$ times faster than the classical method, whereas the indirect strategy is in fact marginally \textit{slower} than it, with a global ratio of $0.91$. This last result deserves a comment. The reference solver is here the discretize-then-optimize scheme of Section \ref{sec:esquema}, whose state and adjoint resolutions reduce, by Proposition \ref{prop:esquema_estados} and  \eqref{eq:adjunto}, to two scalar divisions per time step; solving them costs about the same as one forward pass of a DeepONet with a single input function. Replacing them by network evaluations therefore removes no significant cost, while the iterative loop, the line search and the projection remain. In other words, the indirect strategy pays off only when the underlying solver is expensive compared with an inference pass, which is not the case for an explicitly solvable discrete scheme such as the present one. It is important to point out that these measurements were performed only once on a single computer.
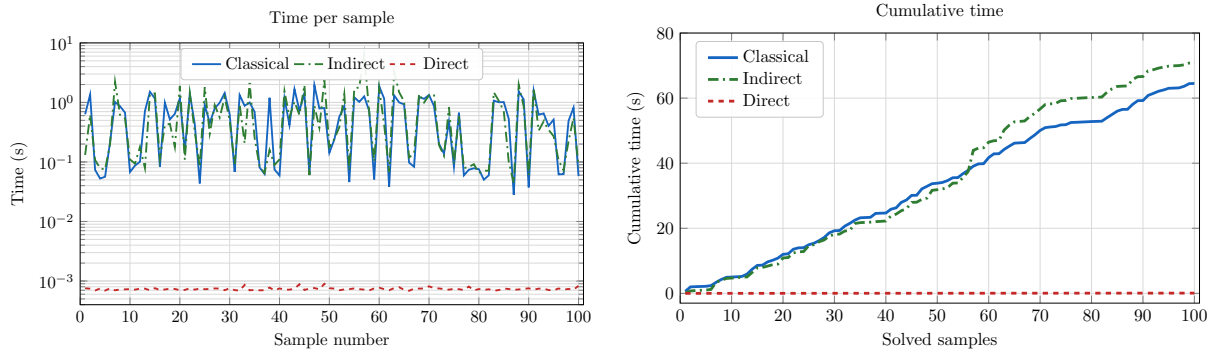
\begin{figure}[H]
\centering
\begin{subfigure}[t]{0.49\textwidth}
\centering
\resizebox{\linewidth}{!}{%
\begin{tikzpicture}
\begin{axis}[
width=12cm, height=7cm,
title={Time per sample},
xmin=0, xmax=101, ymin=4e-4, ymax=10, ymode=log,
xlabel={Sample number}, ylabel={Time (s)},
grid=both, ymajorgrids, grid style={gray!30},
legend cell align=left,
legend style={font=\normalsize, at={(0.5,0.98)}, anchor=north, legend columns=3,
fill=white, fill opacity=0.92, text opacity=1, draw=gray!45,
rounded corners=2pt, inner xsep=6pt, inner ysep=3pt},
/pgf/number format/use comma=false,
/pgf/number format/set decimal separator={.},
tick label style={font=\normalsize}, label style={font=\normalsize},
title style={font=\normalsize}
]
\addplot+[color={rgb,255:red,21;green,101;blue,192}, solid, line width=1.1pt, mark=none]
table[x=sample,y=tclassic] {figuras/tiempos_prueba.dat};
\addlegendentry{Classical}
\addplot+[color={rgb,255:red,46;green,125;blue,50}, dash pattern=on 6pt off 2.5pt on 1.2pt off 2.5pt, line width=1.1pt, mark=none]
table[x=sample,y=t1] {figuras/tiempos_prueba.dat};
\addlegendentry{Indirect}
\addplot+[color={rgb,255:red,198;green,40;blue,40}, dashed, line width=1.1pt, mark=none]
table[x=sample,y=t2] {figuras/tiempos_prueba.dat};
\addlegendentry{Direct}
\end{axis}
\end{tikzpicture}%
}
\end{subfigure}
\hfill
\begin{subfigure}[t]{0.49\textwidth}
\centering
\resizebox{\linewidth}{!}{%
\begin{tikzpicture}
\begin{axis}[
width=12cm, height=7cm,
title={Cumulative time},
xmin=0, xmax=101, ymin=-3, ymax=80,
xlabel={Solved samples}, ylabel={Cumulative time (s)},
grid=both, ymajorgrids, grid style={gray!30},
legend cell align=left,
legend style={font=\normalsize, at={(0.03,0.97)}, anchor=north west,
fill=white, fill opacity=0.92, text opacity=1, draw=gray!45,
rounded corners=2pt, inner xsep=6pt, inner ysep=3pt},
/pgf/number format/use comma=false,
/pgf/number format/set decimal separator={.},
/pgf/number format/1000 sep={},
scaled y ticks=false,
tick label style={font=\normalsize}, label style={font=\normalsize},
title style={font=\normalsize}
]
\addplot+[color={rgb,255:red,21;green,101;blue,192}, solid, line width=1.6pt, mark=none]
table[x=sample,y=cumclassic] {figuras/tiempos_prueba.dat};
\addlegendentry{Classical}
\addplot+[color={rgb,255:red,46;green,125;blue,50}, dash pattern=on 6pt off 2.5pt on 1.2pt off 2.5pt, line width=1.6pt, mark=none]
table[x=sample,y=cum1] {figuras/tiempos_prueba.dat};
\addlegendentry{Indirect}
\addplot+[color={rgb,255:red,198;green,40;blue,40}, dashed, line width=1.6pt, mark=none]
table[x=sample,y=cum2] {figuras/tiempos_prueba.dat};
\addlegendentry{Direct}
\end{axis}
\end{tikzpicture}%
}
\end{subfigure}
\caption{Time needed to solve the control problem for the $100$ samples of the test set: individual time in logarithmic scale (left) and cumulative time (right). In the right panel the curve of the classical method and that of the indirect strategy are practically superimposed, whereas the one of the direct strategy is indistinguishable from the horizontal axis.}
\label{fig:tiempos_prueba_deeponet}
\end{figure}

Table \ref{tab:deeponet_clasico_control} shows some additional statistics on the performance of the classical method against the two strategies based on DeepONets.
\begin{table}[H]
\centering
{\small
\renewcommand{\arraystretch}{1.12}
\begin{tabular}{l|c|c|c}
\hline
\textbf{Quantity} & \textbf{Classical} & \textbf{Indirect} & \textbf{Direct} \\ \hline
Total time ($100$ samples) & $64.506$ s & $70.862$ s & $0.074$ s \\
Average time per sample & $0.645$ s & $0.709$ s & $7.35\times10^{-4}$ s \\
Minimum time & {$0.028$ s} & $0.044$ s & $6.85\times10^{-4}$ s \\
Maximum time & $1.949$ s & $6.528$ s & $8.98\times10^{-4}$ s \\ \hline
Speedup with respect to the classical method & --- & $0.91\times$ & $878\times$ \\ \hline
Average number of iterations & {$161.8$} & {$125.4$} & $1$ \\
Minimum--maximum number of iterations & {$8$--$465$} & {$8$--$801$} & $1$--$1$ \\ \hline
\end{tabular}
}
\caption{Statistics of time and iterations of the three implementations over the test set.}
\label{tab:deeponet_clasico_control}
\end{table}

\noindent
{\footnotesize
\textbf{Acknowledgements. }The first author was supported by the Vicerrector\'ia Acad\'emica of the Universidad Industrial de Santander. Second author thanks Grant I+D+I PID2023-149182NB-I00 funded by MICIU/AEI/10.13039/501100011033 and ERDF/EU and IMUS-Maria de Maeztu grant CEX2024-001517-M - Apoyo a Unidades de Excelencia María de Maeztu, funded by MICIU/AEI/ 10.13039/501100011033.  The third author was supported by the Vicerretoría de Investigación y Extension of the Universidad Industrial de Santander, Proyecto 4838.
}

%%%%%%%%%%%%%%%%%%%%%%%%%%%%%%%%%%%%%%%%%%%%%%%%%%%%%%%%%%%%%
% BIBLIOGRAPHY
%%%%%%%%%%%%%%%%%%%%%%%%%%%%%%%%%%%%%%%%%%%%%%%%%%%%%%%%%%%%%

\newpage
\appendix

\section{Computational setup and construction of the datasets}\label{ap:computacional}
\subsection{Architecture and training of the three DeepONets}\label{ap:setup}

The architecture of the three DeepONets, both for the branch and for the trunk subnetworks, is presented in Table \ref{tab:config_deeponet_edo}; there, the input dimensions of the branch net correspond to the values of the two controls at the $m=400$ equidistant sensors, together with the scalar inputs of each operator, whereas the trunk net receives the queried time. The weights of the linear layers were initialized with a Xavier normal distribution \cite{Glorot} and all the biases were set to zero. Except for the input dimension, which is the only feature distinguishing the three networks, the architecture is identical in both strategies.

\begin{table}[H]
\centering
\begin{tabular}{c|c|c|c}
\hline
 & \multicolumn{2}{c|}{\textbf{\textit{Indirect strategy}}} & \textbf{\textit{Direct strategy}} \\\hline
\textbf{Hyperparameter} & \textbf{States} & \textbf{Adjoint states} & \textbf{Optimal controls} \\ \hline
Input dimension of the \textit{branch} & $802$ & $806$ & $809$ \\
Input dimension of the \textit{trunk} & $1$ & $1$ & $1$ \\
Hidden layers per subnetwork & $3$ & $3$ & $3$ \\
Neurons per hidden layer & $200$ & $200$ & $200$ \\
Latent dimension per component ($\ell$) & $200$ & $200$ & $200$ \\
Output dimension of the \textit{branch} ($2\ell$) & $400$ & $400$ & $400$ \\
Hidden activation & \texttt{tanh} & \texttt{tanh} & \texttt{tanh} \\
Output layer & linear & linear & linear \\
Output components & $u,\ \sigma$ & $w,\ z$ & $\ControlC^*,\ \ControlS^*$ \\
Trainable parameters & $442\,402$ & $443\,202$ & $443\,802$ \\ \hline
\end{tabular}
\caption{Architecture of the three supervised DeepONets.}
\label{tab:config_deeponet_edo}
\end{table}

The amount of data of the datasets used for training and the validation is presented in Table \ref{tab:conjuntos_deeponet_edo}. Each strategy uses its own dataset. In the \textit{indirect strategy}, each sample consists of the controls $[\ControlC,\ControlS]$ evaluated at the sensors, the initial conditions $[u_0,\sigma_0]$, the weights $[k_1,k_2,l_1,l_2]$ of the objective functional which are involved in the equations of the adjoint variables, and the corresponding state variables $[u,\sigma]$ and adjoint variables $[w,z]$, obtained by means of the scheme of Section \ref{sec:esquema}. In the \textit{direct strategy}, each sample consists of the \textit{initial} controls $[\ControlC,\ControlS]$ evaluated at the sensors, the initial conditions $[u_0,\sigma_0]$, the six weights $[k_1,k_2,k_3,k_4,l_1,l_2]$ of the objective functional, the integral constraint $c_{\max}$, and the pair of optimal controls $[\ControlC^{*},\ControlS^{*}]$ produced by the complete classical algorithm. In both cases, \textit{Samples} corresponds to one of those choices of inputs and, for each sample, the corresponding outputs are stored evaluated at $400$ observation points; moreover, \textit{sample--time pairs} represents the total amount of observation points stored in each dataset.

\begin{table}[H]
\centering
\begin{tabular}{c|c|c}
\hline
\textbf{Quantity} & \textbf{Training} & \textbf{Validation} \\ \hline
\multicolumn{3}{c}{\textit{Indirect strategy}} \\ \hline
Samples & $80\,000$ & $20\,000$ \\
Sample--time pairs & $32\,000\,000$ & $8\,000\,000$ \\ \hline
\multicolumn{3}{c}{\textit{Direct strategy}} \\ \hline
Samples & $80\,000$ & $20\,000$ \\
Sample--time pairs & $32\,000\,000$ & $8\,000\,000$ \\ \hline
\end{tabular}
\caption{Composition of the supervised datasets of each strategy.}
\label{tab:conjuntos_deeponet_edo}
\end{table}

Both strategies use datasets of the same size. Generating one sample of the \textit{direct strategy} is nevertheless considerably more expensive, since it requires running the optimal control algorithm completely, whereas one sample of the \textit{indirect strategy} only requires one forward and one backward resolution; this is reflected in the generation times reported below and not in the number of samples.\\

The three networks were trained with the same loss function, which represents the average of the mean squared error of its two output components, evaluated on the sample--time pairs of the corresponding mini-batch. As optimizer we employed Adam \cite{Adam} with an initial learning rate of $3\times10^{-4}$, varying by means of an exponential decay. Additionally, in order to avoid numerical difficulties, we used gradient norm clipping with maximum value equal to $1$ (see \cite{Pascanu}). The data corresponding to the training, which are common to the three networks, are summarized in Table \ref{tab:entrenamiento_deeponet_edo}.

\begin{table}[H]
\centering
\begin{tabular}{c|c}
\hline
\textbf{Parameter} & \textbf{Value}\\ \hline
\shortstack[c]{Samples per training mini-batch.} & $300$ \\
\shortstack[c]{Samples per batch during validation.} & $512$ \\
\shortstack[c]{Maximum number of optimization\\iterations per DeepONet} & $500\,000$ (states and optimal controls); $800\,000$ (adjoint) \\
\shortstack[c]{Minimum number of iterations before\\allowing the stopping.} & $100\,000$ \\
\shortstack[c]{Evaluation of the validation set.} & \shortstack[c]{step $1$ and afterwards every $2\,000$ steps.} \\
\shortstack[c]{Early stopping criterion.} & validation MSE smaller than $10^{-4}$. \\
Selection of the checkpoint. & \shortstack[c]{smallest validation MSE.} \\ \hline
\end{tabular}
\caption{Training and validation data of the DeepONets, common to the three networks.}
\label{tab:entrenamiento_deeponet_edo}
\end{table}

The implementation was carried out in Python (\texttt{3.13.12}); specifically, the library \texttt{PyTorch} (\texttt{2.11.0}) was used for the construction, the training and the inference of the networks, while the classical solver of Section \ref{sec:esquema} was implemented directly in Python. All the computations were run on a single computer, a MacBook Air M4 (2025) with 16 GB of RAM, on which the generation of the data was carried out on CPU, whereas the training and the inference employed the MPS acceleration of PyTorch. It is important to point out that the generation of the training and validation datasets of the \textit{direct strategy} was distributed over $8$ parallel processes, in order to reduce the computational cost, whereas the generation of the datasets of the \textit{indirect strategy} was run sequentially.

\subsection{Construction of the training and validation datasets}\label{ap:entrenamiento}

For each sample, the controls $[\ControlC,\ControlS]$, the initial conditions and the weights of the objective functional were chosen independently; the distributions employed and the temporal discretization are summarized in Table \ref{tab:muestreo_deeponet_edo}. In order that the learned operator was not associated with a single type of control, each function $\ControlC$ and $\ControlS$ was taken from one of the four families of Table \ref{tab:familias_controles_dataset}, chosen at random and independently for each control, and these were mapped onto the range $[0,1]$. The families, their probabilities and their random construction are the same in both strategies; what changes is the role played by the sampled controls, since in the \textit{indirect strategy} they are the controls at which the state and adjoint operators are evaluated, whereas in the \textit{direct strategy} they are the \textit{initial} controls with which the algorithm producing the label starts.

\begin{table}[H]
\centering
{\small
\setlength{\tabcolsep}{7pt}
\renewcommand{\arraystretch}{1.12}
\begin{tabular}{p{7.00cm}|p{6.00cm}}
\hline
\textbf{Quantity} & \textbf{Value} \\ \hline
\multicolumn{2}{c}{\textit{Common to both strategies}} \\ \hline
Time interval & $[0,2]$ \\
Sensors per control & $400$ equidistant points \\
Observation points & $400$ equidistant points \\
Range of $\ControlC$ and $\ControlS$ & $[0,1]$ \\
Distribution of $u_0$ & Continuous uniform on $[10^{-1},4]$ \\
Distribution of $\sigma_0$ & Continuous uniform on $[10^{-3},2]$ \\
Oxygen targets & $\sigma_d=\sigma_{dT}=1$ \\ \hline
\multicolumn{2}{c}{\textit{Indirect strategy}} \\ \hline
$l_1,l_2,k_1,k_2$ & Discrete uniform on $\{0,\dots,20\}$ \\ \hline
\multicolumn{2}{c}{\textit{Direct strategy}} \\ \hline
$l_1,l_2,k_1,k_2,k_3,k_4$ & Discrete uniform on $\{0,\dots,20\}$ \\
Distribution of $c_{\max}$ & Continuous uniform on $[0.05,2)$ \\
Constraint on the initial $\ControlC$ & $\int_0^{2}\ControlC(t)\,dt<c_{\max}$ \\ \hline
\end{tabular}
}
\caption{Discretization and distributions used to build the training data.}
\label{tab:muestreo_deeponet_edo}
\end{table}

\begin{table}[H]
\centering
\renewcommand{\arraystretch}{1.14}
\begin{tabular}{l|c|p{8.60cm}}
\hline
\textbf{Family} & \textbf{Probability} & \textbf{Random construction} \\ \hline
Piecewise constant & $17\%$ & Between $1$ and $8$ intervals of minimal length $0.10$, where the values are taken from a $\operatorname{Beta}(0.8,0.8)$ distribution. With a probability $0.45,$ one of the pieces is forced to be $0$ or $1$. \\[5pt]
Bang-bang & $35\%$ & Between $2$ and $7$ intervals with the same rule for the lengths, where the constants $0$ or $1$ are considered on each interval. \\[5pt]
PCHIP & $30\%$ & Between $4$ and $10$ equispaced nodes are considered, with values $\operatorname{Beta}(0.8,0.8)$. Then, from these nodes, a continuous cubic interpolation is performed. With a probability $0.50$ one node is fixed to $0$ or $1$. \\[5pt]
Sinusoidal & $18\%$ & Functions of the form $0.5+0.5\sin(\pi ft+\phi)$, with $f\in(0.5,4)$ and $\phi\in(0,2\pi)$. \\ \hline
\end{tabular}
\caption{Families of functions used to generate independently the controls $\ControlC$ and $\ControlS$.}
\label{tab:familias_controles_dataset}
\end{table}

In the \textit{indirect strategy} the sampling imposes only the pointwise constraints $0\leq\ControlC,\ControlS\leq1$; the integral constraint was not enforced, with the aim of obtaining a better generalization during the training. This is admissible because the operators $\mathcal{G}_e$ and $\mathcal{G}_a$ are well defined for any pair of bounded controls, by Theorem \ref{teo:existencia_edo} and Proposition \ref{prop:esquema_estados}, and the admissibility is guaranteed afterwards by the projection onto the set of admissible controls inside the optimal control algorithm. In spite of this, $18\,889$ of the $80\,000$ controls $\ControlC$ of the training dataset ($23.61\%$) and $4\,626$ of the $20\,000$ of the validation dataset ($23.13\%$) satisfy the integral constraint $\int_0^{2}\ControlC<0.75$, which corresponds to a representative value of the bound $c_{\max}$.\\

In the \textit{direct strategy}, in contrast, the constraint must indeed be imposed in order to have an admissible initial control. Thus, for each sample a bound $c_{\max}$ was generated and the initial control $\ControlC$ was projected onto the admissible set by means of the procedure described at the end of Section \ref{sec:esquema}, so that $\int_0^{2}\ControlC(t)\,dt<c_{\max}$ holds strictly. By sampling $c_{\max}$ on the whole range $[0.05,2)$, the network learns from scenarios with a very limited total therapy up to scenarios in which the integral constraint is essentially inactive. Moreover, a candidate sample is discarded and replaced whenever the algorithm producing its label does not converge within \texttt{maxIter} iterations or the resulting control fails the validation test \eqref{eq:filtro_general}; this affected $4\,469$ of the $84\,469$ ($5.29\%$) attempted training samples ($748$ non-convergent, $3\,721$ failing the validation test) and $1\,191$ of the $21\,191$ ($5.62\%$) attempted validation samples ($217$ non-convergent, $974$ failing the validation test).\\

Tables \ref{tab:dataset_deeponet_edo} and \ref{tab:dataset_deeponet_directa_edo} summarize the cost of the generation of the datasets of each strategy. For the \textit{indirect strategy} the process was run sequentially on CPU and took in total $2$ min $34.2$ s. For the \textit{direct strategy}, in contrast, $8$ parallel processes were used and even so the process took $5$ h $50$ min $57.3$ s. 

\begin{table}[H]
\centering
{\small
\setlength{\tabcolsep}{4pt}
\begin{tabular}{l|r|r|c}
\hline
\textbf{Dataset} & \textbf{Samples} & \textbf{Time (s)} & \textbf{Size (GB)} \\ \hline
Training & $80\,000$ & $123.0$  & $0.770$ \\
Validation & $20\,000$ & $31.3$ & $0.192$ \\ \hline
Total & $100\,000$ & $154.2$ & $0.962$ \\ \hline
\end{tabular}
}
\caption{Computational costs of the generation of the datasets of the \textit{indirect strategy}.}
\label{tab:dataset_deeponet_edo}
\end{table}

\begin{table}[H]
\centering
{\small
\setlength{\tabcolsep}{4pt}
\begin{tabular}{l|r|r|c}
\hline
\textbf{Dataset} & \textbf{Samples} & \textbf{Time (s)} & \textbf{Size (GB)} \\ \hline
Training & $80\,000$ & $17\,313.5$  & $0.515$ \\
Validation & $20\,000$ & $3\,743.7$ & $0.129$ \\ \hline
Total & $100\,000$ & $21\,057.3$ & $0.644$ \\ \hline
\end{tabular}
}
\caption{Computational costs of the generation of the datasets of the \textit{direct strategy}.}
\label{tab:dataset_deeponet_directa_edo}
\end{table}

\subsection{Training histories of the three networks}\label{ap:historias}

In order to complement the values of Table \ref{tab:errores_deeponet_edo}, Figure \ref{fig:perdidas_deeponet_edo} presents the training and validation curves of the three networks. These curves allow one to verify the evolution of the fitting and to distinguish the behavior of each output component.

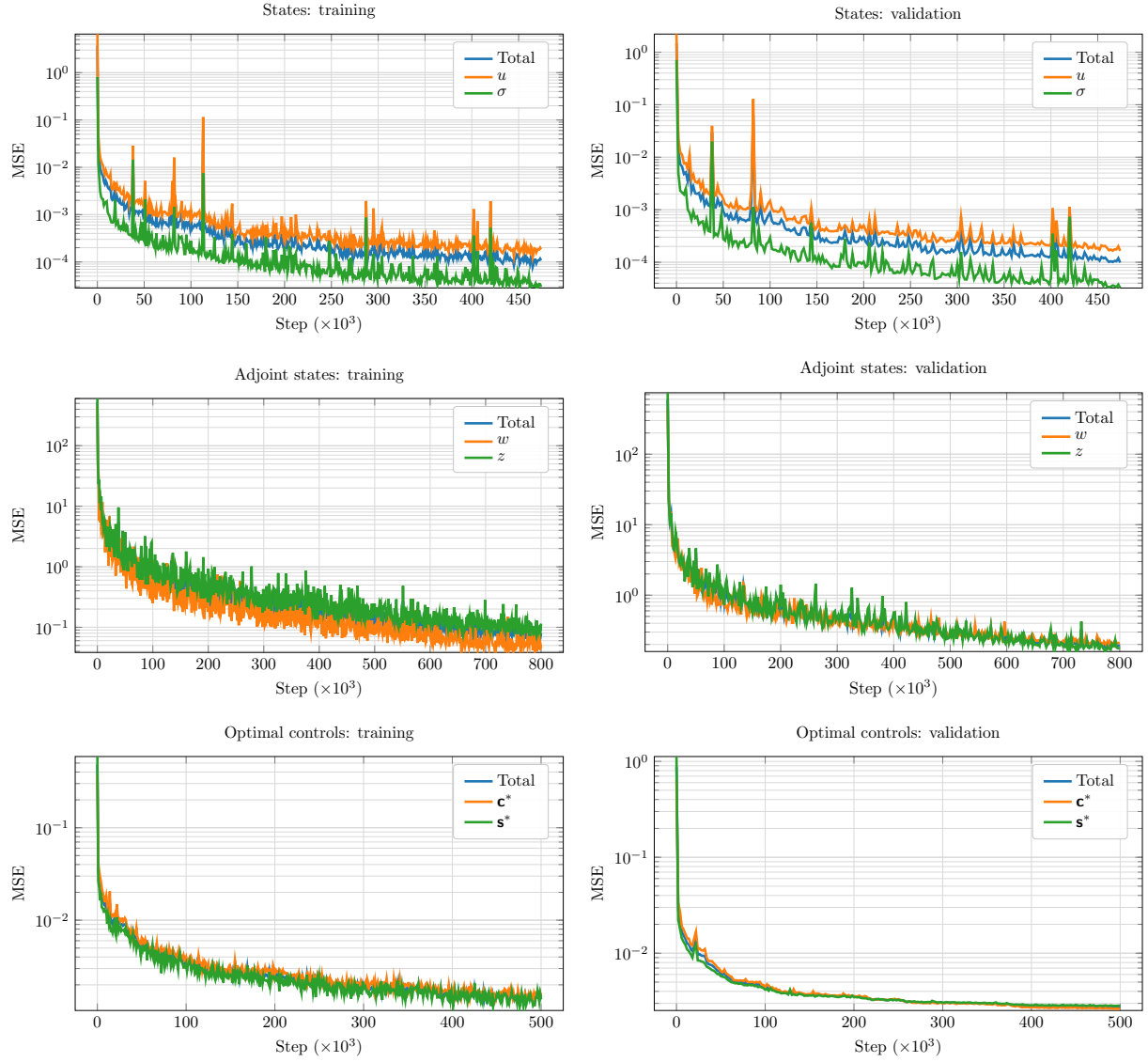
\begin{figure}[H]
\centering
\begin{subfigure}[t]{0.49\textwidth}
\centering
\resizebox{\linewidth}{!}{%
\begin{tikzpicture}
\begin{axis}[
width=12cm, height=7cm,
title={States: training},
xmin=-23700, xmax=497700,
ymin=2.808734e-5, ymax=6.536413,
ymode=log,
xtick={0,50000,100000,150000,200000,250000,300000,350000,400000,450000},
xticklabels={0,50,100,150,200,250,300,350,400,450},
xlabel={Step ($\times 10^{3}$)}, ylabel={MSE},
grid=both, ymajorgrids, grid style={gray!30},
legend cell align=left,
legend style={font=\normalsize, at={(0.97,0.97)}, anchor=north east,
fill=white, fill opacity=0.92, text opacity=1, draw=gray!45,
rounded corners=2pt, inner xsep=6pt, inner ysep=4pt},
/pgf/number format/use comma=false,
/pgf/number format/set decimal separator={.},
/pgf/number format/1000 sep={},
scaled x ticks=false,
tick label style={font=\normalsize},
label style={font=\normalsize},
title style={font=\normalsize}
]
\addplot+[color={rgb,255:red,31;green,119;blue,180}, solid, line width=1.5pt, mark=none]
table[x=step,y=total] {figuras/perdidas_estados_entrenamiento.dat};
\addlegendentry{Total}
\addplot+[color={rgb,255:red,255;green,127;blue,14}, solid, line width=1.5pt, mark=none]
table[x=step,y=u] {figuras/perdidas_estados_entrenamiento.dat};
\addlegendentry{$u$}
\addplot+[color={rgb,255:red,44;green,160;blue,44}, solid, line width=1.5pt, mark=none]
table[x=step,y=sigma] {figuras/perdidas_estados_entrenamiento.dat};
\addlegendentry{$\sigma$}
\end{axis}
\end{tikzpicture}%
}
\end{subfigure}
\hfill
\begin{subfigure}[t]{0.49\textwidth}
\centering
\resizebox{\linewidth}{!}{%
\begin{tikzpicture}
\begin{axis}[
width=12cm, height=7cm,
title={States: validation},
xmin=-23700, xmax=497700,
ymin=3.231240e-5, ymax=2.214394,
ymode=log,
xtick={0,50000,100000,150000,200000,250000,300000,350000,400000,450000},
xticklabels={0,50,100,150,200,250,300,350,400,450},
xlabel={Step ($\times 10^{3}$)}, ylabel={MSE},
grid=both, ymajorgrids, grid style={gray!30},
legend cell align=left,
legend style={font=\normalsize, at={(0.97,0.97)}, anchor=north east,
fill=white, fill opacity=0.92, text opacity=1, draw=gray!45,
rounded corners=2pt, inner xsep=6pt, inner ysep=4pt},
/pgf/number format/use comma=false,
/pgf/number format/set decimal separator={.},
/pgf/number format/1000 sep={},
scaled x ticks=false,
tick label style={font=\normalsize},
label style={font=\normalsize},
title style={font=\normalsize}
]
\addplot+[color={rgb,255:red,31;green,119;blue,180}, solid, line width=1.5pt, mark=none]
table[x=step,y=total] {figuras/perdidas_estados_validacion.dat};
\addlegendentry{Total}
\addplot+[color={rgb,255:red,255;green,127;blue,14}, solid, line width=1.5pt, mark=none]
table[x=step,y=u] {figuras/perdidas_estados_validacion.dat};
\addlegendentry{$u$}
\addplot+[color={rgb,255:red,44;green,160;blue,44}, solid, line width=1.5pt, mark=none]
table[x=step,y=sigma] {figuras/perdidas_estados_validacion.dat};
\addlegendentry{$\sigma$}
\end{axis}
\end{tikzpicture}%
}
\end{subfigure}

\medskip

\begin{subfigure}[t]{0.49\textwidth}
\centering
\resizebox{\linewidth}{!}{%
\begin{tikzpicture}
\begin{axis}[
width=12cm, height=7cm,
title={Adjoint states: training},
xmin=-40000, xmax=840000,
ymin=3.837314e-2, ymax=602.1691,
ymode=log,
xtick={0,100000,200000,300000,400000,500000,600000,700000,800000},
xticklabels={0,100,200,300,400,500,600,700,800},
xlabel={Step ($\times 10^{3}$)}, ylabel={MSE},
grid=both, ymajorgrids, grid style={gray!30},
legend cell align=left,
legend style={font=\normalsize, at={(0.97,0.97)}, anchor=north east,
fill=white, fill opacity=0.92, text opacity=1, draw=gray!45,
rounded corners=2pt, inner xsep=6pt, inner ysep=4pt},
/pgf/number format/use comma=false,
/pgf/number format/set decimal separator={.},
/pgf/number format/1000 sep={},
scaled x ticks=false,
tick label style={font=\normalsize},
label style={font=\normalsize},
title style={font=\normalsize}
]
\addplot+[color={rgb,255:red,31;green,119;blue,180}, solid, line width=1.5pt, mark=none]
table[x=step,y=total] {figuras/perdidas_adjuntos_entrenamiento.dat};
\addlegendentry{Total}
\addplot+[color={rgb,255:red,255;green,127;blue,14}, solid, line width=1.5pt, mark=none]
table[x=step,y=w] {figuras/perdidas_adjuntos_entrenamiento.dat};
\addlegendentry{$w$}
\addplot+[color={rgb,255:red,44;green,160;blue,44}, solid, line width=1.5pt, mark=none]
table[x=step,y=z] {figuras/perdidas_adjuntos_entrenamiento.dat};
\addlegendentry{$z$}
\end{axis}
\end{tikzpicture}%
}
\end{subfigure}
\hfill
\begin{subfigure}[t]{0.49\textwidth}
\centering
\resizebox{\linewidth}{!}{%
\begin{tikzpicture}
\begin{axis}[
width=12cm, height=7cm,
title={Adjoint states: validation},
xmin=-40000, xmax=840000,
ymin=0.1589959, ymax=743.0569,
ymode=log,
xtick={0,100000,200000,300000,400000,500000,600000,700000,800000},
xticklabels={0,100,200,300,400,500,600,700,800},
xlabel={Step ($\times 10^{3}$)}, ylabel={MSE},
grid=both, ymajorgrids, grid style={gray!30},
legend cell align=left,
legend style={font=\normalsize, at={(0.97,0.97)}, anchor=north east,
fill=white, fill opacity=0.92, text opacity=1, draw=gray!45,
rounded corners=2pt, inner xsep=6pt, inner ysep=4pt},
/pgf/number format/use comma=false,
/pgf/number format/set decimal separator={.},
/pgf/number format/1000 sep={},
scaled x ticks=false,
tick label style={font=\normalsize},
label style={font=\normalsize},
title style={font=\normalsize}
]
\addplot+[color={rgb,255:red,31;green,119;blue,180}, solid, line width=1.5pt, mark=none]
table[x=step,y=total] {figuras/perdidas_adjuntos_validacion.dat};
\addlegendentry{Total}
\addplot+[color={rgb,255:red,255;green,127;blue,14}, solid, line width=1.5pt, mark=none]
table[x=step,y=w] {figuras/perdidas_adjuntos_validacion.dat};
\addlegendentry{$w$}
\addplot+[color={rgb,255:red,44;green,160;blue,44}, solid, line width=1.5pt, mark=none]
table[x=step,y=z] {figuras/perdidas_adjuntos_validacion.dat};
\addlegendentry{$z$}
\end{axis}
\end{tikzpicture}%
}
\end{subfigure}

\medskip

\begin{subfigure}[t]{0.49\textwidth}
\centering
\resizebox{\linewidth}{!}{%
\begin{tikzpicture}
\begin{axis}[
width=12cm, height=7cm,
title={Optimal controls: training},
xmin=-24998.95, xmax=524999.95,
ymin=1.056084e-3, ymax=0.588288,
ymode=log,
xtick={0,100000,200000,300000,400000,500000},
xticklabels={0,100,200,300,400,500},
xlabel={Step ($\times 10^{3}$)}, ylabel={MSE},
grid=both, ymajorgrids, grid style={gray!30},
legend cell align=left,
legend style={font=\normalsize, at={(0.97,0.97)}, anchor=north east,
fill=white, fill opacity=0.92, text opacity=1, draw=gray!45,
rounded corners=2pt, inner xsep=6pt, inner ysep=4pt},
/pgf/number format/use comma=false,
/pgf/number format/set decimal separator={.},
/pgf/number format/1000 sep={},
scaled x ticks=false,
tick label style={font=\normalsize},
label style={font=\normalsize},
title style={font=\normalsize}
]
\addplot+[color={rgb,255:red,31;green,119;blue,180}, solid, line width=1.5pt, mark=none]
table[x=step,y=total] {figuras/perdidas_controles_entrenamiento.dat};
\addlegendentry{Total}
\addplot+[color={rgb,255:red,255;green,127;blue,14}, solid, line width=1.5pt, mark=none]
table[x=step,y=C] {figuras/perdidas_controles_entrenamiento.dat};
\addlegendentry{$\ControlC^*$}
\addplot+[color={rgb,255:red,44;green,160;blue,44}, solid, line width=1.5pt, mark=none]
table[x=step,y=S] {figuras/perdidas_controles_entrenamiento.dat};
\addlegendentry{$\ControlS^*$}
\end{axis}
\end{tikzpicture}%
}
\end{subfigure}
\hfill
\begin{subfigure}[t]{0.49\textwidth}
\centering
\resizebox{\linewidth}{!}{%
\begin{tikzpicture}
\begin{axis}[
width=12cm, height=7cm,
title={Optimal controls: validation},
xmin=-24998.95, xmax=524999.95,
ymin=2.529405e-3, ymax=1.122407,
ymode=log,
xtick={0,100000,200000,300000,400000,500000},
xticklabels={0,100,200,300,400,500},
xlabel={Step ($\times 10^{3}$)}, ylabel={MSE},
grid=both, ymajorgrids, grid style={gray!30},
legend cell align=left,
legend style={font=\normalsize, at={(0.97,0.97)}, anchor=north east,
fill=white, fill opacity=0.92, text opacity=1, draw=gray!45,
rounded corners=2pt, inner xsep=6pt, inner ysep=4pt},
/pgf/number format/use comma=false,
/pgf/number format/set decimal separator={.},
/pgf/number format/1000 sep={},
scaled x ticks=false,
tick label style={font=\normalsize},
label style={font=\normalsize},
title style={font=\normalsize}
]
\addplot+[color={rgb,255:red,31;green,119;blue,180}, solid, line width=1.5pt, mark=none]
table[x=step,y=total] {figuras/perdidas_controles_validacion.dat};
\addlegendentry{Total}
\addplot+[color={rgb,255:red,255;green,127;blue,14}, solid, line width=1.5pt, mark=none]
table[x=step,y=C] {figuras/perdidas_controles_validacion.dat};
\addlegendentry{$\ControlC^*$}
\addplot+[color={rgb,255:red,44;green,160;blue,44}, solid, line width=1.5pt, mark=none]
table[x=step,y=S] {figuras/perdidas_controles_validacion.dat};
\addlegendentry{$\ControlS^*$}
\end{axis}
\end{tikzpicture}%
}
\end{subfigure}
\caption{Evolution of the error during the training of the three DeepONets: the state and the adjoint networks, corresponding to the \textit{indirect strategy}, and the optimal-control network, corresponding to the \textit{direct strategy}.}
\label{fig:perdidas_deeponet_edo}
\end{figure}

Figure \ref{fig:distribucion_error_validacion} shows, for each of the three networks, the histogram of the mean squared error calculated in each sample over the validation set, along with its $50$, $90$, and $99$ percentiles. In all three cases, the distribution exhibits a strong right-skew, confirming a nice generalization of the three networks on the validation set.
\newcommand{\imgdoscolhist}{0.47\textwidth}
\pgfplotsset{
estilohistval/.style={
scale only axis,
width=5.72cm, height=3.5cm,
fill=blue!45, draw=blue!75!black,
xlabel={MSE of the sample}, ylabel={Validation samples},
grid=both, grid style={gray!30},
scaled y ticks=false,
yticklabel style={text width=1.05cm, align=right},
tick label style={font=\footnotesize}, label style={font=\small},
}}

\newcommand{\lineamuestraedo}[2]{\tikz[baseline=-0.6ex]{\draw[#1] (0,0)--(0.62,0);}\,#2}
\newcommand{\leyendamuestrasedo}{%
\begingroup\footnotesize
\lineamuestraedo{color={rgb,255:red,21;green,101;blue,192}, line width=1.4pt}{$\ControlC$ classical}\quad
\lineamuestraedo{color={rgb,255:red,46;green,125;blue,50}, line width=1.4pt, dash pattern=on 5pt off 2pt on 1pt off 2pt}{$\ControlC$ indirect}\quad
\lineamuestraedo{color={rgb,255:red,198;green,40;blue,40}, line width=1.4pt, dashed}{$\ControlC$ direct}\par\smallskip
\lineamuestraedo{color={rgb,255:red,106;green,27;blue,154}, line width=1.4pt}{$\ControlS$ classical}\quad
\lineamuestraedo{color={rgb,255:red,0;green,131;blue,143}, line width=1.4pt, dash pattern=on 5pt off 2pt on 1pt off 2pt}{$\ControlS$ indirect}\quad
\lineamuestraedo{color={rgb,255:red,239;green,108;blue,0}, line width=1.4pt, dashed}{$\ControlS$ direct}\par
\endgroup\smallskip}

% #1 subcaption text, #2 data file, #3-5 p50/p90/p99, #6 ymax, #7 xtick list
\newcommand{\panelhistval}[7]{%
\begin{subfigure}[t]{\imgdoscolhist}
\centering
\begin{tikzpicture}
\begin{axis}[estilohistval, xmode=log, ybar interval,
ymin=0, ymax=#6, xtick={#7}]
\addplot table[x=edge,y=count] {figuras/#2};
\draw[magenta, thick, dashed] (axis cs:#3,0) -- (axis cs:#3,#6);
\draw[orange!90!black, thick, dashed] (axis cs:#4,0) -- (axis cs:#4,#6);
\draw[red!70!black, thick, dashed] (axis cs:#5,0) -- (axis cs:#5,#6);
\end{axis}
\end{tikzpicture}
\caption{#1}
\end{subfigure}}

\begin{figure}[H]
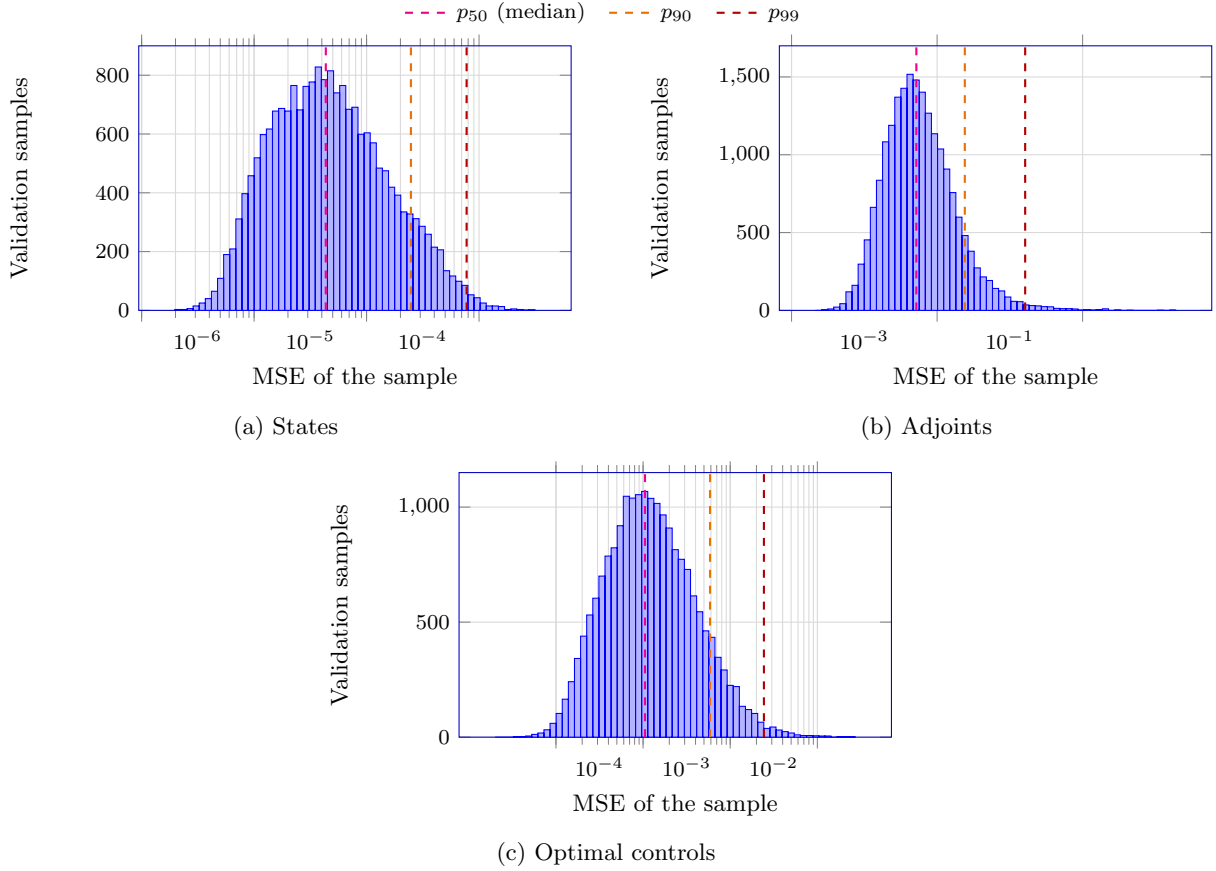

\centering
\begingroup\footnotesize
\lineamuestraedo{magenta, thick, dashed}{$p_{50}$ (median)}\quad
\lineamuestraedo{orange!90!black, thick, dashed}{$p_{90}$}\quad
\lineamuestraedo{red!70!black, thick, dashed}{$p_{99}$}
\endgroup\smallskip

\panelhistval{States}{hist_edges_estados.dat}{4.341e-05}{2.479e-04}{7.756e-04}{900}{1e-6,1e-5,1e-4,1e-3}\hfill
\panelhistval{Adjoints}{hist_edges_adjuntos.dat}{5.188e-02}{2.404e-01}{1.623e+00}{1700}{1e-3,1e-1,1e1}

\medskip
\panelhistval{Optimal controls}{hist_edges_controles.dat}{1.052e-03}{5.870e-03}{2.437e-02}{1150}{1e-4,1e-3,1e-2,1e-1}
\caption{Distribution, over the $20\,000$ samples of the validation set, of the mean squared error attained by each of the three DeepONets. The dashed lines mark the $50$th, $90$th and $99$th percentiles of each distribution.}
\label{fig:distribucion_error_validacion}
\end{figure}

\section{Optimal controls obtained for some samples}\label{ap:samples}

In this appendix, we present the optimal controls obtained by both strategies for some samples chosen at random from the test set of $100$ problems. Specifically, samples $3$, $18$, $47$, $63$, $81$, and $84$ were selected. For each of these samples, Figure \ref{fig:muestras_apendice} shows the optimal controls $\ControlC$ and $\ControlS$ in the first row, and the absolute error of the two strategies relative to the classical reference (for $\ControlC$ and for $\ControlS$) in the second row. Table \ref{tab:muestras_prueba_deeponet} supplements the figure with the mean squared error, the maximum pointwise error, and the number of iterations required by the classical method and the \textit{indirect strategy} for each of the six samples (the \textit{direct strategy} always requires a single evaluation of the network).

\newcommand{\imgdoscol}{0.47\textwidth}
\pgfplotsset{
estilomuestraedofila/.style={
scale only axis,
width=6.47cm, height=3.25cm,
xmin=0, xmax=2,
xlabel={$t$}, ylabel={},
grid=both, grid style={gray!30},
/pgf/number format/fixed,
/pgf/number format/precision=2,
/pgf/number format/use comma=false,
/pgf/number format/set decimal separator={.},
scaled y ticks=false,
yticklabel style={/pgf/number format/fixed, /pgf/number format/precision=2, text width=0.7cm, align=right},
tick label style={font=\small}, label style={font=\small},
title style={font=\small, yshift=-1mm}
}}

% #1 data file, #2 label of the sample, #3 error bound in C, #4 error bound in S
% 2x2 grid per sample: row 1 = control C, control S; row 2 = error in C, error in S.
\newcommand{\panelmuestraedo}[4]{%
\noindent\footnotesize\textbf{#2}\par\smallskip
\begin{subfigure}[t]{\imgdoscol}
\centering
\begin{tikzpicture}
\begin{axis}[estilomuestraedofila, title={Control $\ControlC$},
ymin=-0.09, ymax=1.09]
\addplot+[color={rgb,255:red,21;green,101;blue,192}, solid, line width=1.6pt, mark=none]
table[x=t,y=C0] {figuras/#1};
\addplot+[color={rgb,255:red,46;green,125;blue,50}, dash pattern=on 6pt off 2.5pt on 1.2pt off 2.5pt, line width=1.6pt, mark=none]
table[x=t,y=C1] {figuras/#1};
\addplot+[color={rgb,255:red,198;green,40;blue,40}, dashed, line width=1.6pt, mark=none]
table[x=t,y=C2] {figuras/#1};
\end{axis}
\end{tikzpicture}%
\end{subfigure}\hfill
\begin{subfigure}[t]{\imgdoscol}
\centering
\begin{tikzpicture}
\begin{axis}[estilomuestraedofila, title={Control $\ControlS$},
ymin=-0.09, ymax=1.09]
\addplot+[color={rgb,255:red,106;green,27;blue,154}, solid, line width=1.6pt, mark=none]
table[x=t,y=S0] {figuras/#1};
\addplot+[color={rgb,255:red,0;green,131;blue,143}, dash pattern=on 6pt off 2.5pt on 1.2pt off 2.5pt, line width=1.6pt, mark=none]
table[x=t,y=S1] {figuras/#1};
\addplot+[color={rgb,255:red,239;green,108;blue,0}, dashed, line width=1.6pt, mark=none]
table[x=t,y=S2] {figuras/#1};
\end{axis}
\end{tikzpicture}%
\end{subfigure}

\medskip
\begin{subfigure}[t]{\imgdoscol}
\centering
\begin{tikzpicture}
\begin{axis}[estilomuestraedofila, title={Error in $\ControlC$},
ymin=0, ymax=#3]
\addplot+[color={rgb,255:red,46;green,125;blue,50}, dash pattern=on 6pt off 2.5pt on 1.2pt off 2.5pt, line width=1.6pt, mark=none]
table[x=t,y=errC1] {figuras/#1};
\addplot+[color={rgb,255:red,198;green,40;blue,40}, dashed, line width=1.6pt, mark=none]
table[x=t,y=errC2] {figuras/#1};
\end{axis}
\end{tikzpicture}%
\end{subfigure}\hfill
\begin{subfigure}[t]{\imgdoscol}
\centering
\begin{tikzpicture}
\begin{axis}[estilomuestraedofila, title={Error in $\ControlS$},
ymin=0, ymax=#4]
\addplot+[color={rgb,255:red,0;green,131;blue,143}, dash pattern=on 6pt off 2.5pt on 1.2pt off 2.5pt, line width=1.6pt, mark=none]
table[x=t,y=errS1] {figuras/#1};
\addplot+[color={rgb,255:red,239;green,108;blue,0}, dashed, line width=1.6pt, mark=none]
table[x=t,y=errS2] {figuras/#1};
\end{axis}
\end{tikzpicture}%
\end{subfigure}%
}

\begin{figure}[H]
\centering
\leyendamuestrasedo

\panelmuestraedo{muestra_3.dat}{Sample 3}{0.121}{0.180}

\medskip
\panelmuestraedo{muestra_18.dat}{Sample 18}{0.124}{0.233}
\caption{Optimal controls and absolute error for six samples of the test set selected at random (1/3).}
\label{fig:muestras_apendice}
\end{figure}

\clearpage
\begin{figure}[H]
\ContinuedFloat
\centering

\panelmuestraedo{muestra_37.dat}{Sample 37}{0.463}{0.632}

\medskip
\panelmuestraedo{muestra_47.dat}{Sample 47}{0.121}{0.196}
\caption{Optimal controls and absolute error for six samples of the test set selected at random (2/3).}
\end{figure}

\clearpage
\begin{figure}[H]
\ContinuedFloat
\centering

\panelmuestraedo{muestra_63.dat}{Sample 63}{0.111}{0.129}

\medskip
\panelmuestraedo{muestra_81.dat}{Sample 81}{0.573}{0.035}
\caption{Optimal controls and absolute error for six samples of the test set selected at random (3/3).}
\end{figure}

\begin{table}[H]
\centering
{\small
\setlength{\tabcolsep}{4pt}
\renewcommand{\arraystretch}{1.12}
\begin{tabular}{c|cc|cc|cc}
\hline
 & \multicolumn{2}{c|}{\textbf{MSE}} & \multicolumn{2}{c|}{\textbf{Maximum error}} & \multicolumn{2}{c}{\textbf{Iterations}}  \\[-2mm]
\textbf{Sample} & $\ControlC$ & $\ControlS$ & $\ControlC$ & $\ControlS$ & Classical & Strategy \\ \hline
\multicolumn{7}{c}{\textit{Indirect strategy}} \\ \hline
$3$ & $1.06\times10^{-3}$ & $5.80\times10^{-4}$ & $0.107$ & $0.086$ & $37$ & $41$  \\
$18$ & $1.56\times10^{-3}$ & $8.12\times10^{-4}$ & $0.118$ & $0.120$ & $128$ & $76$  \\
{$37$} & {$1.03\times10^{-2}$} & {$2.14\times10^{-2}$} & {$0.441$} & {$0.458$} & {$29$} & {$16$}  \\
$47$ & $8.22\times10^{-4}$ & $1.36\times10^{-3}$ & $0.106$ & $0.187$ & $456$ & $184$  \\
$63$ & $3.37\times10^{-4}$ & $4.59\times10^{-4}$ & $0.095$ & $0.121$ & $306$ & $519$  \\
$81$ & $9.60\times10^{-3}$ & $2.14\times10^{-5}$ & $0.546$ & $0.034$ & $26$ & $18$  \\ \hline
\multicolumn{7}{c}{\textit{Direct strategy}} \\ \hline
$3$ & {$1.90\times10^{-3}$} & {$2.57\times10^{-3}$} & {$0.115$} & {$0.172$} & $37$ & $1$ \\
$18$ & {$2.97\times10^{-4}$} & {$2.30\times10^{-3}$} & {$0.063$} & {$0.222$} & $128$ & $1$ \\
{$37$} & {$1.86\times10^{-2}$} & {$4.90\times10^{-2}$} & {$0.311$} & {$0.602$} & {$29$} & $1$ \\
$47$ & {$9.58\times10^{-4}$} & {$1.46\times10^{-3}$} & {$0.116$} & {$0.109$} & $456$ & $1$ \\
$63$ & {$6.27\times10^{-4}$} & {$6.25\times10^{-4}$} & {$0.105$} & {$0.123$} & $306$ & $1$ \\
$81$ & {$2.95\times10^{-3}$} & {$4.82\times10^{-5}$} & {$0.318$} & {$0.024$} & $26$ & $1$  \\ \hline
\end{tabular}
}
\caption{Discrepancy of each strategy with respect to the classical method on the six represented samples.}
\label{tab:muestras_prueba_deeponet}
\end{table}


\begin{thebibliography}{99}

\bibitem{BarryStraume2026}
Barry-Straume, J.; Sarshar, A.; Popov, A. A. \& Sandu, A.
(2026).
Physics-informed neural networks for PDE-constrained optimization and control.
\textit{Communications on Applied Mathematics and Computation}, \textbf{8}, 1283--1306.

\bibitem{Boulle}
Boullé, N. \& Townsend, A.
(2023).
A mathematical guide to operator learning.
\textit{arXiv:2312.14688}.

\bibitem{ChenChen}
Chen, T. \& Chen, H.
(1995).
Universal approximation to nonlinear operators by neural networks with arbitrary activation functions and its application to dynamical systems.
\textit{IEEE Transactions on Neural Networks}, \textbf{6}(4), 911--917.

\bibitem{Coddington} Coddington, E. A. \& Levinson, N. (1955). Theory of Ordinary Differential Equations. McGraw-Hill, New York.

\bibitem{Demo2023}
Demo, N.; Strazzullo, M. \& Rozza, G.
(2023).
An extended physics informed neural network for preliminary analysis of parametric optimal control problems.
\textit{Computers \& Mathematics with Applications}, \textbf{143}, 383--396.

\bibitem{Glorot}
Glorot, X. \& Bengio, Y.
(2010).
Understanding the difficulty of training deep feedforward neural networks.
\textit{Proceedings of the Thirteenth International Conference on Artificial Intelligence and Statistics}, PMLR \textbf{9}, 249--256.

\bibitem{Gomez2017}
Gómez, H.
(2017).
Quantitative analysis of the proliferative-to-invasive transition of hypoxic glioma cells.
\textit{Integrative Biology}, \textbf{9}, 257--262.

\bibitem{Goodfellow}
Goodfellow, I.; Bengio, Y. \& Courville, A.
(2016).
\textit{Deep Learning}.
Cambridge, Massachusetts: MIT Press.

\bibitem{Hinze}
Hinze, M.; Pinnau, R.; Ulbrich, M. \& Ulbrich, S.
(2009).
\textit{Optimization with PDE Constraints}.
Dordrecht: Springer.

\bibitem{Hwang2022}
Hwang, R.; Lee, J. Y.; Shin, J. Y. \& Hwang, H. J.
(2022).
Solving PDE-constrained control problems using operator learning.
\textit{Proceedings of the AAAI Conference on Artificial Intelligence}, \textbf{36}(4), 4504--4512.

\bibitem{Adam}
Kingma, D. P. \& Ba, L. J.
(2015).
Adam: A method for stochastic optimization.
\textit{International Conference on Learning Representations (ICLR)}.

\bibitem{Leo2}
López-Agredo, J. L.; Rueda-Gómez, D. A. \& Villamizar-Roa, É. J.
(2023).
Numerical analysis of a mathematical model describing the evolution of hypoxic glioma cells.
\textit{Computers \& Mathematics with Applications}, \textbf{131}, 138--157.

\bibitem{LuDeepONet}
Lu, L.; Jin, P.; Pang, G.; Zhang, Z. \& Karniadakis, G. E.
(2021).
Learning nonlinear operators via DeepONet based on the universal approximation theorem of operators.
\textit{Nature Machine Intelligence}, \textbf{3}, 218--229.

\bibitem{luenberger} Luenberger, D. G. (1969). \textit{Optimization by Vector Space Methods}. New York: John Wiley \& Sons.

\bibitem{Mowlavi2023}
Mowlavi, S. \& Nabi, S.
(2023).
Optimal control of PDEs using physics-informed neural networks.
\textit{Journal of Computational Physics}, \textbf{473}, 111731.

\bibitem{Pascanu}
Pascanu, R.; Mikolov, T. \& Bengio, Y.
(2013).
On the difficulty of training recurrent neural networks.
\textit{Proceedings of the 30th International Conference on Machine Learning}, PMLR \textbf{28}(3), 1310--1318.

\bibitem{Polak}
Polak, E.
(1997).
\textit{Optimization: Algorithms and Consistent Approximations}.
New York: Springer.

\bibitem{Raissi}
Raissi, M.; Perdikaris, P. \& Karniadakis, G. E.
(2019).
Physics-informed neural networks: A deep learning framework for solving forward and inverse problems involving nonlinear partial differential equations.
\textit{Journal of Computational Physics}, \textbf{378}, 686--707.

\bibitem{Troltzsch}
Tröltzsch, F.
(2010).
\textit{Optimal Control of Partial Differential Equations: Theory, Methods and Applications}.
Graduate Studies in Mathematics \textbf{112}. Providence: American Mathematical Society.

\bibitem{Yin2024}
Yin, P.; Xiao, G.; Tang, K. \& Yang, C.
(2024).
AONN: An adjoint-oriented neural network method for all-at-once solutions of parametric optimal control problems.
\textit{SIAM Journal on Scientific Computing}, \textbf{46}(1), C127--C153.

\bibitem{Yong2024}
Yong, J.; Luo, X. \& Sun, S.
(2024).
Deep multi-input and multi-output operator networks method for optimal control of PDEs.
\textit{Electronic Research Archive}, \textbf{32}(7), 4291--4320.


\end{thebibliography}
\end{document}